\documentclass[
final
]{dmtcs-episciences}

\usepackage[utf8]{inputenc}
\usepackage{subfigure}

\author{Neal Madras\affiliationmark{1}\thanks{Supported in part by a Discovery Grant from NSERC Canada}
  \and Justin M. Troyka\affiliationmark{2}\thanks{Supported by funding from the Faculty of Science at York University towards postdoctoral position}}
\title[Bounded affine permutations I\@. Pattern avoidance and enumeration]{Bounded affine permutations I. \\ Pattern avoidance and enumeration}
\affiliation{
  Department of Mathematics and Statistics, York University, Canada\\
  Department of Mathematics and Computer Science, Davidson College, United States}
\received{2020-03-03}
\revised{2021-03-01}
\accepted{2021-03-13}

\usepackage{amsmath,amssymb,amsthm,tikz,mathdots}

\everymath{\displaystyle}

\newcommand{\sfrac}[2]{\textstyle\frac{#1}{#2}\displaystyle}
\newcommand{\negativespace}{\vspace{-1pc}}

\newcommand{\Z}{\mathbb{Z}}
\newcommand{\R}{\mathbb{R}}
\newcommand{\C}{\mathbb{C}}
\newcommand{\q}[1]{\widetilde{#1}}
\newcommand{\w}[1]{a}
\newcommand{\E}{\widetilde{S}^{/\!/}}
\newcommand{\Av}{S}
\newcommand{\AvA}{\widetilde{S}}
\newcommand{\AvBA}{\E}

\newcommand{\CC}{\mathcal{C}}
\newcommand{\aff}[1]{{\oplus#1}}
\newcommand{\ind}[1]{#1^{\not\oplus}}
\newcommand{\first}{a}
\newcommand{\OO}{\mathcal{O}}

\newcommand{\gr}[1]{\textup{gr}(#1)}
\newcommand{\upgr}[1]{\overline{\textup{gr}}(#1)}
\newcommand{\logr}[1]{\underline{\textup{gr}}(#1)}

\newtheorem{thm}{Theorem}
\newtheorem{prop}[thm]{Proposition}
\newtheorem{lem}[thm]{Lemma}
\newtheorem{cor}[thm]{Corollary}
\newtheorem{conj}[thm]{Conjecture}

\theoremstyle{definition}
\newtheorem{defn}[thm]{Definition}
\newtheorem{expl}[thm]{Example}
\newtheorem{rem}[thm]{Remark}

\newcounter{i}

\newcommand{\drawpermutation}[3][1]{\begin{tikzpicture}[scale=0.5,baseline=(O.base)]
\setcounter{i}{0}
\foreach \j in {#2} {
\stepcounter{i}
\draw (0.5*#1,\value{i}*#1) -- (#3*#1+0.5*#1,\value{i}*#1);
\draw (\value{i}*#1,0.5*#1) -- (\value{i}*#1,#3*#1+0.5*#1);
\node at (\value{i}*#1, 0) {\footnotesize$\j$};
\draw[fill] (\value{i}*#1, \j*#1) circle (0.2);
}
\node (O) at (#3*0.5*#1,#3*0.5*#1) {};
\end{tikzpicture}}

\newcommand{\drawpattern}[4][1]{\begin{tikzpicture}[scale=0.5,baseline=(O.base)]
\foreach \x in {1,...,#3} {
\draw (0.5*#1,\x*#1) -- (#3*#1+0.5*#1,\x*#1);
\draw (\x*#1,0.5*#1) -- (\x*#1,#3*#1+0.5*#1);
}
\setcounter{i}{0}
\foreach \j in {#2} {
\stepcounter{i}
\node at (\value{i}*#1, 0) {\footnotesize$\j$};
\draw[fill] (\value{i}*#1, \j*#1) circle (0.2);
\foreach \k in {#4} {
\ifnum \j=\k
\draw [thick] (\value{i}*#1, \j*#1) circle (0.4);
\fi
}
}
\node (O) at (#3*0.5*#1,#3*0.5*#1) {};
\end{tikzpicture}}

\begin{document}
\publicationdetails{22}{2021}{2}{1}{6178}
\maketitle
\begin{abstract}
We introduce a new boundedness condition for affine permutations, motivated by the fruitful concept of periodic boundary conditions in statistical physics.
We study pattern avoidance in bounded affine permutations.
In particular, we show that if $\tau$ is one of the finite
increasing oscillations, then every $\tau$-avoiding affine
permutation satisfies the boundedness condition. We also explore the enumeration of pattern-avoiding affine permutations that can be decomposed into blocks, using analytic methods to relate their exact and asymptotic enumeration to that of the underlying ordinary permutations. Finally, we perform exact and asymptotic enumeration of the set
of all bounded affine permutations of size $n$.
A companion paper will focus on avoidance of monotone decreasing patterns in bounded affine permutations.
\end{abstract}

\section{Introduction}
   \label{sec-intro}
Pattern-avoiding permutations have been studied actively in the combinatorics literature for the past four decades.  
(See Section \ref{sec-definitions} for definitions of terms we use.)
Some sources on permutation patterns include: \cite{Bevan} for essential terminology, \cite[Ch.\ 4]{BonaCP} for a textbook introduction, and \cite{VatterSurvey} for an in-depth survey of the literature.
Pattern-avoiding permutations arise in a variety of mathematical contexts, particularly algebra and the analysis of algorithms.
Research such as \cite{Crites, BilleyCrites} have extended these investigations by considering
affine permutations that avoid one or more (ordinary) permutations as patterns.

\begin{defn} 
   \label{def.affine}
An \emph{affine permutation of size $n$} is a bijection $\sigma \colon \Z \to \Z$ such that:
\begin{enumerate}[(i)]
\item $\sigma(i+n) \,=\, \sigma(i)\,+\, n$ for all 
$i \in \Z$, and 
\item $\sum_{i=1}^n \sigma(i) = \sum_{i=1}^n i$.
\end{enumerate}\end{defn}

Condition (ii) 
can be viewed as a ``centering'' condition, since any
bijection satisfying (i) can be made to satisfy (ii) by adding a constant to the function.
The affine permutations of size $n$ form an infinite Coxeter group under composition, with $n$ generators; see Section 8.3 of Bj\"orner and Brenti \cite{BB} for a detailed look at affine permutations from this perspective.

For any given size $n>1$, there are infinitely many affine permutations of size $n$; indeed, for some patterns such as $\tau=321$, there are infinitely many affine permutations of size $n$ that avoid $\tau$.  One can view the following definition
as a reasonable attempt to make these sets finite, but as we describe below and in Section \ref{sec-motiv}, there are more compelling reasons for considering this
definition.

\begin{defn} 
   \label{def.bounded}
A \emph{bounded affine permutation} of size $n$ is an affine permutation $\sigma$ of size $n$ such that $|\sigma(i) - i| < n$ for all $i$. \end{defn}

Figure \ref{fig:boundedaffine} illustrates an example of a bounded affine permutation.

\begin{figure}
\[
\begin{tikzpicture}[scale=0.25]
\draw (1,5) [fill=black] circle (.3);
\draw (2,2) [fill=black] circle (.3);
\draw (3,4) [fill=black] circle (.3);
\draw (4,9) [fill=black] circle (.3);
\draw (5,0) [fill=black] circle (.3);
\draw (6,1) [fill=black] circle (.3);
\draw (7,11) [fill=black] circle (.3);
\draw (8,8) [fill=black] circle (.3);
\draw (9,10) [fill=black] circle (.3);
\draw (10,15) [fill=black] circle (.3);
\draw (11,6) [fill=black] circle (.3);
\draw (12,7) [fill=black] circle (.3);
\draw (13,17) [fill=black] circle (.3);
\draw (14,14) [fill=black] circle (.3);
\draw (15,16) [fill=black] circle (.3);
\draw (16,21) [fill=black] circle (.3);
\draw (17,12) [fill=black] circle (.3);
\draw (18,13) [fill=black] circle (.3);
\draw (19,23) [fill=black] circle (.3);
\draw (20,20) [fill=black] circle (.3);
\draw (21,22) [fill=black] circle (.3);
\draw (22,27) [fill=black] circle (.3);
\draw (23,18) [fill=black] circle (.3);
\draw (24,19) [fill=black] circle (.3);
\draw [thick] (-2,8) -- (27,8);
\draw [thick] (8,-2) -- (8,27);
\node at (-1,-0.5) {$\iddots$};
\node at (26,26) {$\iddots$};
\draw [dashed,thick] (-2,4) -- (21,27);
\draw [dashed,thick] (4,-2) -- (27,21);
\end{tikzpicture} \]
\caption{A bounded affine permutation of size $6$, whose values on $1,\ldots,6$ are $2, 7, -2, -1, 9, 6$. For the affine permutation to be a bounded affine permutation of size $6$, its entries must all lie strictly between the dashed lines.}
\label{fig:boundedaffine}
\end{figure}
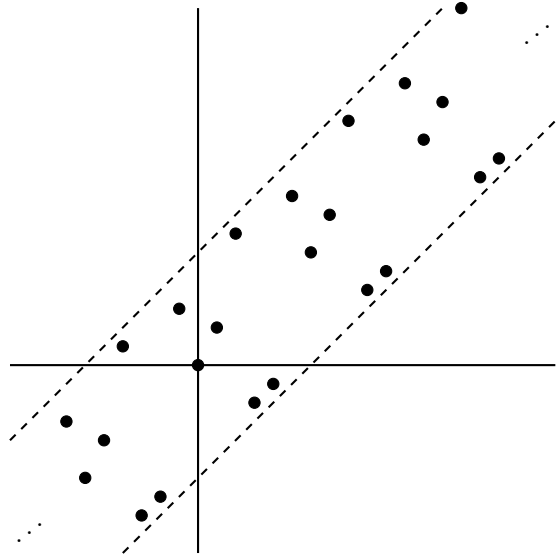

\begin{rem}
Affine permutations with a different boundedness condition were introduced by Knutson, Lam, and Speyer \cite{KLS}, who used them to study the totally non-negative Grassmannian and positroids. The bounded affine permutations in our paper are not the same as those.
\end{rem}

\begin{defn} \label{def:S}
Let $S_n$ denote the set of permutations of size $n$, let $\widetilde{S}_n$ denote the set of affine permutations of size $n$, and let $\E_n$ denote the set of bounded affine permutations of size $n$.
We also define
\[
S \;:=\;  \bigsqcup_{n\ge0}S_n
     \hspace{5mm}\hbox{and}\hspace{5mm}
     \widetilde{S}   \;:=\;  \bigsqcup_{n\ge1}\widetilde{S}_n\,
     \hspace{5mm}\hbox{and}\hspace{5mm}
     \E   \;:=\;  \bigsqcup_{n\ge1}\E_n\,. \]
\end{defn}

\begin{rem} \label{rem:disjoint1}
If $d \mid n$, then every element of $\widetilde{S}_d$ is also an element of $\widetilde{S}_n$. If $\omega$ is an affine permutation of size $n$, then $n$ need not be the smallest possible size of $\omega$. The notation in Definition \ref{def:S} uses a disjoint union so that each element of $\widetilde{S}$ comes equipped with a ``size''. As a result, an affine permutation in $\widetilde{S}_n$ and the same permutation considered as an element of $\widetilde{S}_{kn}$ (for $k \ge 2$) are different objects in $\widetilde{S}$, distinguished by their size. Our definition using disjoint union is natural for enumeration purposes: our count of affine permutations of size $n$ with a given property does not require that $n$ is the smallest possible size.
\end{rem}

\begin{rem} \label{rem:disjoint2}
Of course $\E_n$ is a proper subset of $\widetilde{S}_n$ for each $n$, and so $\E$ is a proper subset of $\widetilde{S}$. In contrast, every element of $\widetilde{S}_n$ is an element of $\E_{kn}$ for some $k$, so $\bigcup_{n\ge1} \E_n = \bigcup_{n\ge1} \widetilde{S}_n$. Using disjoint union in Definition \ref{def:S} means that not every affine permutation (equipped with a size) is a bounded affine permutation.
\end{rem}

If we view a permutation $\pi\in S_n$ as a
bijection on $[n]$, then we can extend it periodically by 
Equation (i) of Definition \ref{def.affine}
to a bijection $\oplus\pi$ on $\mathbb{Z}$; that is,
\[ \oplus \pi(i+kn) = \pi(i) + kn \quad \text{for $i \in [n]$ and $k \in \Z$}. \]
Observe that $\oplus \pi \in \E_n$ (see 
Figure \ref{fig.affinepi}). We call $\oplus \pi$ the \emph{infinite sum} of $\pi$. The map $\pi \mapsto \oplus \pi$ is an injection from $S_n$ into $\E_n$.

\setlength{\unitlength}{1.2mm}
\begin{figure}
  \begin{center}
%
\begin{picture}(70,70)(-25,-25)
\put(-23,0){\vector(1,0){68}}
\put(0,-23){\vector(0,1){68}}
\put(2,2){\line(1,0){20}}
\put(2,2){\line(0,1){20}}
\put(2,22){\line(1,0){20}}
\put(22,2){\line(0,1){20}}
\put(2,0){\line(0,1){1}}
\put(1,-3){$1$}
\put(22,0){\line(0,1){1}}
\put(20,-3){$n$}
\put(0,2){\line(1,0){1}}
\put(-3,1){$1$}
\put(0,22){\line(1,0){1}}
\put(-3,21){$n$}
\put(13,8){\huge{${\pi}$}}
\put(2,8){\circle*{2}}
\put(4,2){\circle*{2}}
\put(6,16){\circle*{2}}
\put(8,22){\circle*{2}}
\put(22,6){\circle*{2}}
%
\put(-22,-2){\line(1,1){47}}
\put(-2,-22){\line(1,1){47}}
\put(24,0){\line(0,1){1}}
\put(24,-3){$n{+}1$}
\put(42,0){\line(0,1){1}}
\put(40,-3){$2n$}
\put(24,24){\line(1,0){20}}
\put(24,24){\line(0,1){20}}
\put(24,44){\line(1,0){20}}
\put(44,24){\line(0,1){20}}
\put(26,30){\large{copy of ${\pi}$}}
\put(24,30){\circle*{2}}
\put(26,24){\circle*{2}}
\put(28,38){\circle*{2}}
\put(30,44){\circle*{2}}
\put(44,28){\circle*{2}}
\put(-20,-20){\line(1,0){20}}
\put(-20,-20){\line(0,1){20}}
\put(-20,0){\line(1,0){20}}
\put(0,-20){\line(0,1){20}}
\put(-18,-14){\large{copy of ${\pi}$}}
\put(-20,-14){\circle*{2}}
\put(-18,-20){\circle*{2}}
\put(-16,-6){\circle*{2}}
\put(-14,0){\circle*{2}}
\put(0,-16){\circle*{2}}
%
%
\end{picture}
\caption{Schematic plot of a permutation $\pi\in S_n$ and 
its periodic extension $\oplus \pi \in \E_n$.  For an affine
permutation of size $n$ to be bounded, all points of the 
plot must lie on or between the two diagonal lines.
\label{fig.affinepi}}  
  \end{center}
\end{figure}
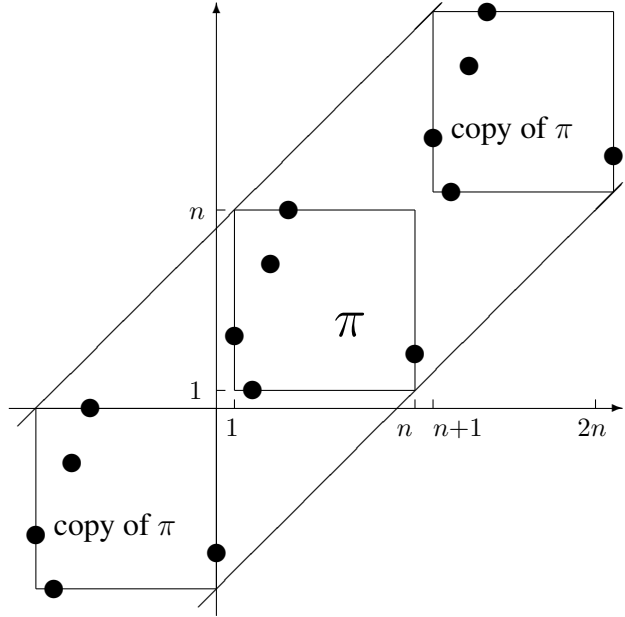

Most of this paper concerns the set of affine permutations or bounded affine permutations that avoid some fixed set of permutations $R$; these sets are denoted $\AvA(R)$ and $\AvBA(R)$ respectively, and we define pattern avoidance and related notions in Section \ref{sec-definitions}.

Let $\tau \in S_k$. It is routine to check that, if $\tau_1>\tau_k$ (or more generally if $\tau$ is sum-indecomposable), then $\sigma\oplus \pi$ avoids $\tau$ whenever $\sigma$ and $\pi$ both avoid $\tau$.
Thus the injection $\pi\mapsto \oplus\pi$ mentioned 
above is also an injection from $S_n(\tau)$ into $\E_n(\tau)$.
This proves that $|S_n(\tau)|\,\leq\,|\E_n(\tau)|$.
It is harder to find a good general upper bound for $|\E_n(\tau)|$.
We pose the following conjecture.

\begin{conj}
    \label{conj1}
The proper growth rate $\lim_{n\rightarrow\infty}|\E_n(\tau)|^{1/n}$ exists and equals the Stanley--Wilf limit $L(\tau):=\lim_{n\rightarrow\infty}|S_n(\tau)|^{1/n}$
for every sum-indecomposable pattern $\tau$.  
\end{conj}

We remark that the indecomposability condition in the conjecture is 
important; e.g.\ the only affine permutation that avoids $2143$ is the identity permutation.
We can prove that the conjecture holds for some 
specific choices of $\tau$, and that it holds when $\tau$ is an increasing oscillation (see Section \ref{sec:recognizing}), but in general we cannot
even prove that the proper growth rate $\gr{\E(\tau)}$ exists.
At least
we can show that the upper growth rate $\upgr{\E(\tau)}$ is
always finite by the following easy argument.

\begin{prop}
   \label{prop.3L}
Let $\tau$ be a pattern.  Then
$\limsup_{n\rightarrow\infty}|\E_n(\tau)|^{1/n} \,\leq \,3L(\tau),$ 
where $L(\tau)$ is the Stanley--Wilf limit.
\end{prop}
\begin{proof}
Observe that for any $\sigma\in\E_n$, the set of
integers $M(\sigma):=\{\sigma(i):i\in[n]\}$ must all be
distinct modulo $n$.  The boundedness condition tells us 
that $M(\sigma)\subseteq (-n,2n)$, so there are at most $3^n$
possible sets that $M(\sigma)$ could be as $\sigma$ 
varies over $\E_n$.

Given $\sigma\in\E_n(\tau)$, let $\sigma^{\dagger}$ be the
ordinary permutation in $S_n(\tau)$ 
that is order-isomorphic to the string $(\sigma(1),\sigma(2),\ldots,\sigma(n))$.
Then $M(\sigma)$ and $\sigma^{\dagger}$ determine $\sigma$:
indeed, the value of $\sigma(i)$ must be the 
$[\sigma^{\dagger}(i)]^\text{th}$ smallest element of $M(\sigma)$.  It follows that $|\E_n(\tau)|\leq 3^n|S_n(\tau)|$.
\end{proof}

Certain pattern-avoiding affine permutation classes exhibit a curious phenomenon: every affine permutation in the class is ``really'' just an ordinary finite permutation, in the sense that it is a diagonal shift of an infinite sum of an ordinary permutation (see Figure \ref{fig:decomposable} for an example). An affine permutation of this form is called \emph{decomposable}, and an affine permutation class in which every element has this form is called \emph{decomposable} as well. In Section \ref{sec:recognizing}, we characterize the pattern-avoiding affine permutation classes that are decomposable: they are exactly the ones that do not have the infinite increasing oscillation $\OO$ as an element (see Theorem \ref{thm:oscillation}).

\begin{figure}
\begin{center}
\begin{tikzpicture}[scale=0.25]
\draw (1,2) [fill=black] circle (.3);
\draw (2,4) [fill=black] circle (.3);
\draw (3,3) [fill=black] circle (.3);
\draw (4,1) [fill=black] circle (.3);
\draw (5,6) [fill=black] circle (.3);
\draw (6,5) [fill=black] circle (.3);
\draw (7,8) [fill=black] circle (.3);
\draw (8,10) [fill=black] circle (.3);
\draw (9,9) [fill=black] circle (.3);
\draw (10,7) [fill=black] circle (.3);
\draw (11,12) [fill=black] circle (.3);
\draw (12,11) [fill=black] circle (.3);
\draw (13,14) [fill=black] circle (.3);
\draw (14,16) [fill=black] circle (.3);
\draw (15,15) [fill=black] circle (.3);
\draw (16,13) [fill=black] circle (.3);
\draw (17,18) [fill=black] circle (.3);
\draw (18,17) [fill=black] circle (.3);
\draw (19,20) [fill=black] circle (.3);
\draw (20,22) [fill=black] circle (.3);
\draw (21,21) [fill=black] circle (.3);
\draw (22,19) [fill=black] circle (.3);
\draw (23,24) [fill=black] circle (.3);
\draw (24,23) [fill=black] circle (.3);
\draw [thick] (-2,8) -- (27,8);
\draw [thick] (8,-2) -- (8,27);
\draw [thick,dashed] (0.5,0.5) rectangle (6.5,6.5);
\draw [thick,dashed] (6.5,6.5) rectangle (12.5,12.5);
\draw [thick,dashed] (12.5,12.5) rectangle (18.5,18.5);
\draw [thick,dashed] (18.5,18.5) rectangle (24.5,24.5);
\node at (-1,-0.5) {$\iddots$};
\node at (26,26) {$\iddots$};
\end{tikzpicture}
\end{center}
\caption{This affine permutation of size $6$ is decomposable, because it is the infinite direct sum of the permutation $243165$, shifted diagonally downwards by $2$ units.}
\label{fig:decomposable}
\end{figure}
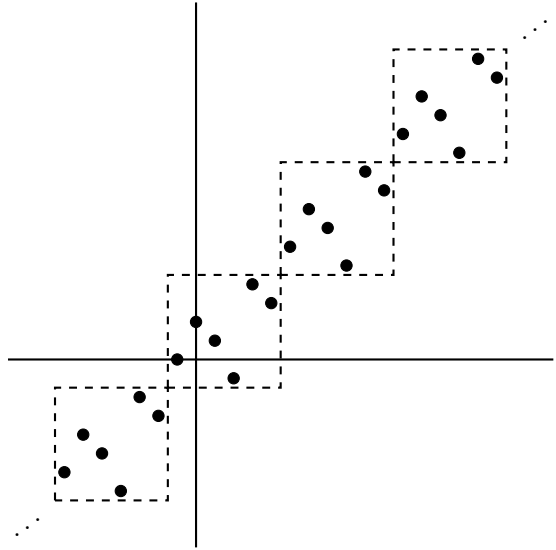

Decomposable classes are particularly easy to enumerate in terms of the corresponding class of ordinary permutations. In fact, regardless of whether $\AvA(R)$ is decomposable, we can enumerate the affine permutations in $\AvA(R)$ that are decomposable (assuming $R$ is a set of indecomposable permutations). These are just the diagonal shifts of infinite sums of (ordinary) permutations that avoid $R$; the set of affine permutations that arise in this way is denoted $\aff{\Av(R)}$. Sections \ref{sec:exact}, \ref{sec:schema}, and \ref{sec:asympt} explore the enumerative relationship between $\aff{\CC}$ and $\CC$ for a sum closed class $\CC$ of (ordinary) permutations. We find a simple formula relating the generating function $\widetilde{F}(x)$ for $\aff{\CC}$ and the generating function $F(x)$ for $\CC$ (Theorem \ref{thm:exactclass}):
\begin{equation} \widetilde{F}(x) = x\frac{d}{dx}\log(F(x)). \label{eq:exactclassintro} \end{equation}
In the process of proving this, we also prove an inequality (Proposition \ref{prop:exact}) showing that $|\CC_n| \le |\aff{C}_n| \le n\,|\CC_n|$. In particular, this means that the growth rates are the same.

The form of Equation \ref{eq:exactclassintro} is significant, appearing in many different guises throughout combinatorics and probability. There is a cluster of related ideas concerning the arrangement of combinatorial objects into a cycle structure. The enumeration of these cycles is linked to that of the constituent objects via Equation \ref{eq:exactclassintro}. Many instances of this relation involve the enumeration of various kinds of trees, or various kinds of paths that do not go below the axis (such as Dyck paths), as in the work of Banderier and Flajolet \cite[Sec.\ 4.1]{BF}. The same ideas are at play in many combinatorial proofs of the Lagrange Inversion Formula, first by Raney \cite{Raney} and later by Gessel \cite{Gessel}. The elementary fact underlying these phenomena is the Cycle Lemma of Dvoretsky and Motzkin \cite{DM}, along with its generalization by Spitzer \cite{Spitzer}; for details, see Cori \cite{Cori} and Dershowitz and Zaks \cite{DZ}.

In Sections \ref{sec:schema} and \ref{sec:asympt}, we extract from Equation \ref{eq:exactclassintro} an asymptotic enumeration of $\aff{\CC}_n$ in terms of the $\CC_n$, dependent on the status of $\CC$ as either a subcritical or supercritical sequence schema (defined in Section \ref{sec:schema}). If the class is subcritical and satisfies certain analytic properties, then $|\aff{\CC}_n|$ is asymptotically a constant times $n|\CC_n|$ (Theorem \ref{prop:subasympt}). If the class is supercritical, then $|\aff{\CC}_n|$ is asymptotically a constant times $|\CC_n|$; intriguingly, if $L$ is the growth rate of $\CC$, then $|\aff{\CC}_n|$ is asymptotically $L^n$, with no constant multiplier required (Theorem \ref{prop:superasympt}).

In a companion paper \cite{MT2}, we focus on the case of
avoiding monotone decreasing patterns $\bf{m(m-1)\cdots 321}$
in bounded affine permutations.  More specifically, for 
every $m\geq 3$ we show that there is an explicit constant 
$K_m$ such that 
\[    |\E_n({\bf m(m-1)\cdots 321})|  \;\sim\;
    K_m n^{(m-2)/2}(m-1)^{2n}   \hspace{5mm}\hbox{as }
    n\rightarrow\infty.
\]
We also obtain scaling limits as $n\to\infty$ of random 
elements of $\E_n({\bf m(m-1)\cdots 321})$, in the framework
of convergence of random measures that has been used in the context of permutons \cite{Glebov}.

We conclude the present paper by returning to the set of all 
bounded affine permutations of size $n$. We count them exactly in Section \ref{sec-first} 
and asymptotically in Section \ref{sec-asymp}, resulting in 
Theorems \ref{thm:exacttotal} and \ref{thm:Enasym} which respectively say
\begin{equation}
    \label{eq.Enexact}
    |\E_n| \;=\; \sum_{m=0}^n \binom{n}{m} \sum_{k=0}^m \binom{m}{n-k} (-1)^{n-m} a(m,k)
\end{equation}
where $a(m,k)$ are the Eulerian numbers, and 
\begin{equation}
    \label{eq.Enasym}
  |\E_n| \;\sim \; 
  \sqrt{\frac{3}{2\pi e n}}\, 2^n \,n! \hspace{5mm}
    \hbox{ as }n\rightarrow\infty.
\end{equation}.

\subsection{Motivation}
   \label{sec-motiv}

The plots of large randomly generated pattern-avoiding permutations offer 
visual representations of the 
structure of these permutations.  A particularly tantalizing case is that of  \textbf{4231}-avoidance, where 
a canoe-like shape appears (see Figure \ref{fig:4231}).  
\begin{figure}
\[
    \includegraphics[scale=0.7]{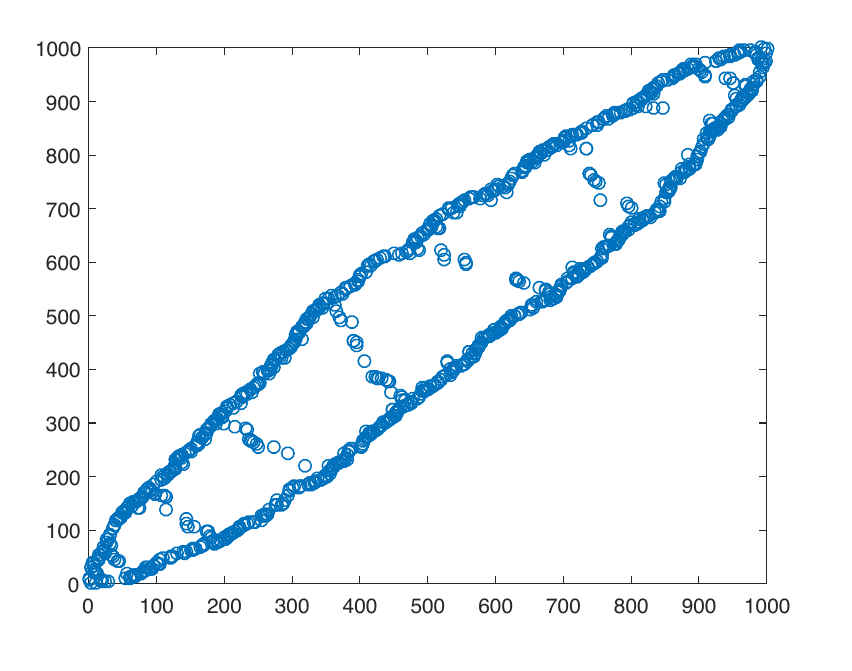}\]
    \caption{A random 4231-avoiding permutation
    of size 1000.  This was generated by Yosef Bisk and
    the first author using
    a Markov chain Monte Carlo algorithm.}
    \label{fig:4231}
\end{figure}
Visually, the middle part of the canoe 
looks roughly like two parallel lines that are joined by more tenuous perpendicular segments (as a canoe's gunwales are joined by thwarts).   
The parallelism of the ``gunwales'' is distorted 
near the canoe's ends, since the plot must fit into a square.   
Following a suggestion of Nathan Clisby that 
uses intuition borrowed from statistical physics, 
perhaps the ends of our \textbf{4231}-avoiding permutation are a ``boundary effect''
that we can try to separate from the main part of the permutation, far from the ends.
Is there a way to make rigorous sense of this idea, particularly at the level of asymptotics when the size
tends to infinity?   Specifically, 
Clisby asked if we could find a way to use \emph{periodic boundary conditions} for 
large pattern-avoiding permutations.

A classical statistical physics model is the Ising model for magnetism.  (See \cite{Thomp} for a nice introduction for 
mathematicians.)
Think of an atom at each point of a regular lattice, which has a ``spin'' that 
is either ``up'' or ``down.''  
For each $n\in \mathbb{N}$, let $\Lambda_n$ be the set of points of the
lattice $\mathbb{Z}^2$ in the square $[1,n]^2$.  
For a given temperature, there is a particular probability distribution on 
the spins of all the atoms  $\Lambda_n$, in which spins at neighbouring sites
interact; if the number of up spins differs significantly from the number down spins, 
then magnetization occurs.  This kind of model typically needs to be studied 
numerically in order to understand the physically relevant case of large $n$.  
But in a large finite square, the behaviour near the center can be rather different
from the behaviour near the boundary, where atoms have fewer neighbours with which to 
interact.  One way to study the behaviour near the center is to impose periodic boundary
conditions on the model, i.e.\ the sites on the right edge of the square are treated as
neighbours of sites on the left edge, and similarly for the top and bottom edges.  
We can think of this as repeating a single configuration of ups and downs periodically 
over $\mathbb{Z}^2$:  the spin at sites $(i,j)$ equals the spin at $(i+an,j+bn)$ for all
integers $a$ and $b$.  With periodic boundary conditions, the behaviour at any site is probabilistically 
the same as at every other site, 
and presumably is similar to behaviour in the middle of a large square without
periodic boundary conditions.

A different statistical physics model is the self-avoiding walk (SAW) \cite{deG,MS}, 
which models the configurations of long polymer molecules.
A SAW on a lattice such as $\mathbb{Z}^d$ is a finite path that 
visits no site more than once.  A randomly chosen long SAW
is believed to look like a fractal; however the behaviour near the ends 
is likely rather different from the behaviour near the middle.  
Clisby \cite{Clis} defined an ``endless self-avoiding walk'' as a SAW that could be
extended infinitely by repeated end-to-end concatenation 
to itself without 
self-intersections.  Not all SAWs can be extended in this way, but for those that
can, we get a model in which all parts of the periodically extended walk
behave like all other parts --- free from any endpoint effects, and more 
representative of the middle of a long (ordinary) SAW.

We take an analogous approach to pattern-avoiding permutations, creating 
``endless pattern-avoiding permutations.''  Any given permutation $\pi$ of size $n$
can be concatenated to itself repeatedly as in Figure \ref{fig.affinepi}, 
corresponding to the periodicity equation 
$\pi(i+an)=\pi(i)+an$ for all integers $i$ and $a$.  
To get something really new, we should allow other permutations of $\mathbb{Z}$
that share this periodicity property---that is, the affine permutations.  
However, even with the periodicity fixed, the spatial variation in the
plots of all affine permutations of size $n$ is too great.
To restrict the class as much as possible,
we impose the boundedness condition $|\sigma(i)-i|<n$, which can be viewed as the 
tightest translation-invariant bound that includes $\oplus \pi$ for every $\pi\in S_n$.

\subsection{Definitions and notation}
\label{sec-definitions}

For sequences $\{a_n\}$ and $\{b_n\}$, we write 
$a_n\sim b_n$ to mean $\lim_{n\to\infty} a_n/b_n \,=\, 1$.

Affine permutations and bounded affine permutations were defined above. We let $S_n$ denote the set of permutations of size $n$, we let $\widetilde{S}_n$ denote the set of affine permutations of size $n$, and we let $\E_n$ denote the set of bounded affine permutations of size $n$. Furthermore, we set $S = \bigsqcup_{n\ge0} S_n$ and $\widetilde{S} = \bigsqcup_{n\ge1} \widetilde{S}_n$ and $\E = \bigsqcup_{n\ge1} \E_n$. For $n \in \mathbb{N}$, we write $[n] = \{1, \ldots, n\}$.

We begin by introducing concepts that are standard in permutation patterns research. The \emph{diagram} or
\emph{plot} of a permutation $\pi \in S_n$ is the set of points $\{(i, \pi(i))\,:\,i \in [n]\}$. Given permutations $\pi$ and $\tau$, we say that \emph{$\pi$ contains $\tau$ as a pattern}, or simply that \emph{$\pi$ contains $\tau$}, if the diagram of $\tau$ can be obtained by deleting zero or more points from the diagram of $\pi$ (and shrinking corresponding segments of the axes), i.e.\ if $\pi$ has a subsequence whose entries have the same relative order as the entries of $\sigma$. We may also say that two sequences with the same relative order are \emph{order-isomorphic}. We say \emph{$\pi$ avoids $\tau$} if $\pi$ does not contain $\tau$. For instance, for $\pi = 493125876$, the subsequence $9356$ is an occurrence of $\tau = 4123$, but on the other hand $\pi$ avoids $3142$. See Figure \ref{fig:contain}.

\begin{figure}
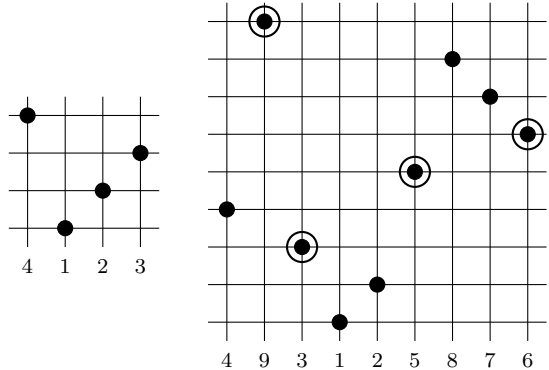

\[ \drawpermutation{4,1,2,3}{4} \hspace{0.25in}
        \drawpattern{4,9,3,1,2,5,8,7,6}{9}{9,3,5,6} \]
    \caption{The permutation $4123$ is contained in the permutation $493125876$.}
    \label{fig:contain}
\end{figure}

A \emph{permutation class} is a set $\CC$ of permutations such that, if $\pi \in \CC$ and $\tau$ is contained in $\pi$, then $\tau \in \CC$. For a permutation class $\CC$, we write $\CC_n = \CC \cap S_n$. If $R$ is a set of permutations, then $\Av{(R)}$ denotes the set of all permutations that avoid every element of $R$. Then $\Av{(R)}$ is a permutation class, and in fact every permutation class is equal to $\Av(R)$ for some set $R$.

Given $\sigma \in S_a$ and $\tau \in S_b$, the \emph{sum} $\sigma \oplus \tau$ is the permutation in $S_{a+b}$ obtained by juxtaposing the diagrams of $\sigma$ and $\tau$ diagonally: that is, $(\sigma \oplus \tau)(i) = \sigma(i)$ if $1 \le i \le a$, and $(\sigma \oplus \tau)(i) = a+\tau(i-a)$ if $a+1 \le i \le a+b$. (This explains the notation $\oplus \sigma$ defined above, which is the doubly infinite sum of $\sigma$ with itself.) A class $\CC$ is \emph{sum closed} if $\sigma, \tau \in \CC$ implies $\sigma \oplus \tau \in \CC$. A permutation is \emph{sum-indecomposable}, or \emph{indecomposable}, if it is not the sum of two permutations of non-zero size. The set of indecomposable permutations in a class $\CC$ is denoted $\ind{\CC}$.

The \emph{upper growth rate} of a permutation class $\CC$ is defined as $\upgr{\CC}:= \limsup_{n\to\infty} |\CC_n|^{1/n}$, and the \emph{lower growth rate} is defined as 
$\logr{\CC}:=\liminf_{n\to\infty} |\CC_n|^{1/n}$. If the upper and lower growth rates of $\CC$ are equal, \textit{i.e.}\ if $\lim_{n\to\infty} |\CC_n|^{1/n}$ exists (or is $\infty$), then this number is called the \emph{proper growth rate} of $\CC$, written $\gr{\CC}$. 
By the Marcus--Tardos Theorem (formerly the Stanley--Wilf Conjecture), every permutation class has a finite upper growth rate except the class of all permutations \cite{MarTar}. It is also known that every sum closed class has a proper growth rate (essentially due to Arratia \cite{Arr}). 
In particular, for each $\tau\in S$, since either $\tau$ is sum-indecomposable or its reverse is, the class $\Av(\tau)$
has a proper growth rate, often called the Stanley--Wilf limit and denoted $L(\tau)$.
It is widely believed that every permutation class has a proper growth rate, but we will refer to the upper or lower growth rate unless we know for sure.

We now introduce the analogous concepts for affine permutations. The \emph{diagram} or \emph{plot} of an affine permutation $\omega \in \widetilde{S}_n$ is the set of points $\{(i, \omega(i)) \,:\,i \in \Z\}$. Given an affine permutation $\omega$ and an ordinary permutation $\tau$, we say that \emph{$\omega$ contains $\tau$ as a pattern}, or simply that \emph{$\omega$ contains $\tau$}, if the diagram of $\tau$ can be obtained by deleting some points from the diagram of $\omega$, i.e.\ if $\omega$ has a subsequence whose entries have the same relative order as the entries of $\tau$. We say \emph{$\pi$ avoids $\tau$} if $\pi$ does not contain $\tau$.

The idea of an affine permutation containing or avoiding a given ordinary permutation was first used by Crites \cite{Crites}, but we can also define one affine permutation containing or avoiding another affine permutation --- this will be important in Section \ref{sec:recognizing}, along with the notion of an affine permutation class (see below). Given $\omega, \omega' \in \widetilde{S}$, we say $\omega$ \emph{contains} $\omega'$ if the diagram of $\omega'$ can be obtained by deleting entries from the diagram of $\omega$ (and stretching the axes appropriately). In case the infinite nature of the diagrams makes this intuitive definition inadequate, we provide a more careful definition: $\omega$ \emph{contains} $\omega'$ if there are two injective order-preserving functions $\phi, \phi' \colon \Z \to \Z$ such that $\omega(\phi(i)) = \phi'(\omega'(i))$ for all $i \in \Z$. We say $\omega$ \emph{avoids} $\omega'$ if $\omega$ does not contain $\omega'$.

If $\omega$ contains $\omega'$, then it is always possible to find a periodic occurrence of $\omega'$ in $\omega$, i.e.\ the functions $\phi, \phi'$ above can be chosen such that, for some $m$, $\phi(i+m) = \phi(i)+m$ and $\phi'(i+m) = \phi'(i)+m$ for all $i$. This fact is tedious to prove and unimportant to this paper, so we omit the proof. Given $\omega \in \widetilde{S}$ and $\tau \in S$, $\omega$ contains the ordinary permutation $\tau$ if and only if $\omega$ contains the affine permutation $\aff{\tau}$.

An \emph{affine permutation class} is a set $\CC$ of affine permutations such that, if $\omega \in \CC$ and $\omega'$ is an affine permutation contained in $\omega$, then $\omega' \in \CC$. For an affine permutation class $\CC$, we write $\CC_n = \CC \cap \widetilde{S}_n$. If $R \subseteq S \cup \widetilde{S}$, then $\AvA{(R)}$ denotes the set of all affine permutations that avoid every element of $R$. Then $\AvA{(R)}$ is an affine permutation class, and every affine permutation class $\CC$ is equal to $\AvA{(R)}$ for some $R$ (for instance we can take $R = \widetilde{S} \smallsetminus \CC$).

Similarly, a \emph{bounded affine permutation class} is a set $\CC$ of bounded affine permutations such that, if $\omega \in \CC$ and $\omega'$ is a bounded affine permutation contained in $\omega$, then $\omega' \in \CC$. The definitions and observations from the previous paragraph apply to bounded affine permutation classes as well. We can define $\upgr{\CC}$, $\logr{\CC}$, and $\gr{\CC}$ for affine permutation classes and bounded affine permutation classes in the same way as for ordinary permutation classes, though we do not know whether $\gr{\E(\tau)}$ exists for every ordinary permutation $\tau$, as it does in the setting of ordinary permutation classes.

We encounter a problem with enumeration: many affine permutation classes have infinitely many elements of each size. On this, there is a result by Crites \cite{Crites} that $\AvA(\tau)$ satisfies $\AvA_n(\tau) < \infty$ for all $n$ if and only if $\tau$ avoids $321$ --- or, more generally, an affine permutation class $\CC$ satisfies $|\CC_n| < \infty$ for all $n$ if and only if $\AvA(321)$ is not a subset of $\CC$. (Our Theorem \ref{thm:oscillation} is a result with a similar flavor.) Of course, a bounded affine permutation class has only finitely many elements of each size.

We will occasionally write an affine permutation $\omega$ in \emph{two-line notation}:
\[ \omega = \begin{pmatrix} \cdots & 1 & 2 & \cdots & n & \cdots \\ \cdots & \omega(1) & \omega(2) & \cdots & \omega(n) & \cdots \end{pmatrix}. \]

\section{Decomposable affine permutation classes} \label{sec:exact}

For $\sigma \in \widetilde{S}_n$ and $r \in \Z$, define $\Sigma^r \sigma \colon \Z \to \Z$ by:
\[ \Sigma^r \sigma(i) = \sigma(i-r) + r. \]
Observe that $\Sigma^r \sigma \in \widetilde{S}_n$. We call $\Sigma^r \sigma$ a \emph{shift} of $\sigma$ (by $r$).

Given a permutation class $\CC$, let $\aff{\CC}$ denote the following set of affine permutations:
\[ \aff{\CC} = \{ \Sigma^r(\oplus \pi) \colon \text{$\pi \in \CC$ and $r \in \Z$}\}. \]
That is, $\aff{\CC}$ is the set of all shifts of infinite sums of permutations in $\CC$. An affine permutation that is a shift of an infinite sum of a permutation shall be called a \emph{decomposable} affine permutation.

Let $\CC$ be a sum closed permutation class. If a permutation $\pi$ is contained as a pattern in some $\sigma \in \aff{\CC}$, then $\pi \in \CC$. As a result, if the permutations in $\CC$ all avoid a given $\pi \in S_n$, then the affine permutations in $\aff{\CC}$ all avoid $\pi$ as well. Furthermore, if an affine permutation $\tau$ is contained as a pattern in some $\sigma \in \aff{\CC}$, then $\tau \in \aff{\CC}$. These facts make it natural to study $\aff{\CC}$ for sum closed $\CC$. In this section and in Sections \ref{sec:schema} and \ref{sec:asympt}, we explore the relationship between $|\CC_n|$ and $|\aff{\CC}_n|$ for sum closed $\CC$. We prove both exact and asymptotic results.

Let $f_n = |\CC_n|$ and $\widetilde{f}_n = |\aff{\CC}_n|$ and $g_n = |\ind{\CC}_n|$. We use the convention that $f_0 = 1$ and $\widetilde{f}_0 = 0$ and $g_0 = 0$. Define
\[ F(x) = \sum_{n\ge0} f_n x^n \quad \text{and} \quad \widetilde{F}(x) = \sum_{n\ge1} \widetilde{f}_n x^n \quad \text{and} \quad G(x) = \sum_{n\ge1} g_n x^n. \]
Thus $F(x)$, $\widetilde{F}(x)$, and $G(x)$ are the ordinary generating functions for, respectively, $\CC$, $\aff{\CC}$, and $\ind{\CC}$.

Since $\CC$ is sum closed, we have the relation
\begin{equation} \label{eqn:seqrelation} F(x) = \frac{1}{1-G(x)} \end{equation}
because each permutation in $\CC$ has a unique decomposition as a sequence of indecomposable permutations in $\CC$.

For each $\pi \in S_n$, let $\first(\pi)$ denote the size of the first indecomposable block of $\pi$, and let $\chi(\pi)$ denote the number of indecomposable blocks of $\pi$. For example, $\first(4312576) = 4$ because the first block of $4312576$ is $4312$, and $\chi(4312576) = 3$ because $4312576$ decomposes into blocks as $4312 \oplus 1 \oplus 21$.

\begin{lem} \label{lem:pair} For $n \ge 1$, $\widetilde{f}_n = |\aff{\CC}_n| = \sum_{\pi \in \CC_n} \first(\pi)$. \end{lem}

\begin{proof}
The sum on the right side is equal to the number of pairs $(\pi, r)$ with $\pi \in \CC_n$ and $r \in \{1, \ldots, \first(\pi)\}$. The set of such pairs is mapped into $\oplus \CC_n$ by $(\pi, r) \mapsto \Sigma^{-r}(\oplus \pi)$, and this mapping is a bijection.
\end{proof}

\begin{prop} \label{prop:exact} Let $n \ge 1$. \begin{enumerate}[(a)]
\item $\widetilde{f}_n = \sum_{k=1}^n kg_kf_{n-k}$.
\item $\max\{ng_n, f_n\} \le \widetilde{f}_n \le nf_n$.
\end{enumerate} \end{prop}
\begin{proof} \let\qqed\qed \let\qed\negativespace 
Using Lemma \ref{lem:pair},
\[ \widetilde{f}_n = \sum_{\pi \in \CC_n} a(\pi) = \sum_{k=1}^n \sum_{\substack{\pi \in \CC_n \\ a(\pi) = k}} k = \sum_{k=1}^n k \cdot \#\{\pi \in \CC_n \colon a(\pi) = k\},  \]
and the number of $\pi \in \CC_n$ whose first block has size $k$ is given by $g_k f_{n-k}$, because there are $g_k$ options for the first block and $f_{n-k}$ options for the rest of $\pi$. This proves (a).

We can now use (a) to prove (b):
\begin{align*}
\sum_{k=1}^n kg_kf_{n-k} &\le \sum_{k=1}^n ng_kf_{n-k} = n\sum_{k=1}^n g_kf_{n-k} = nf_n, \\
\text{and}\quad \sum_{k=1}^n kg_kf_{n-k} &\ge \max\{ng_n, f_n\}. \tag*{\qqed} \end{align*}
\end{proof}

\begin{rem} Proposition \ref{prop:exact}(b) implies that $\underline{\textup{gr}}(\aff{\CC}) = \underline{\textup{gr}}(\CC)$ and $\overline{\textup{gr}}(\aff{\CC}) = \overline{\textup{gr}}(\CC)$. Furthermore, the fact that $\widetilde{f}_n \le nf_n$ stands in contrast to the situation of bounded affine permutations: the number of bounded affine permutations of size $n$ avoiding $321$ is asymptotically $\sfrac{1}{2} n^2\left|\Av_n(321)\right|$. \end{rem}

\begin{thm} \label{thm:exactclass} $\widetilde{F}(x) = x \,G'(x)\,F(x) = \frac{x F'(x)}{F(x)} = x \frac{d}{dx} \log(F(x))$.
\end{thm}
\begin{proof} The first equation, $\widetilde{F}(x) = x \,G'(x)\,F(x)$, follows immediately from Proposition \ref{prop:exact}(a). Next, by \eqref{eqn:seqrelation}, we have $G(x) = 1 - \frac{1}{F(x)}$, so $G'(x) = \frac{F'(x)}{F(x)^2}$, and from this we easily obtain $x\,G'(x)\,F(x) = \frac{x F'(x)}{F(x)}$. The last equation, $\frac{x F'(x)}{F(x)} = x \frac{d}{dx}\log(F(x))$, is merely an identity of formal power series.
\end{proof}

\begin{expl} \label{ex:exact}
\begin{itemize}
\item If $\CC$ is one of $\Av(321)$, $\Av(312)$, or $\Av(231)$, then $F(x) = \frac{1-\sqrt{1-4x}}{2x}$ (this class is famously counted by the Catalan numbers). Applying Theorem \ref{thm:exactclass} gives us $\widetilde{F}(x) = \frac{1}{2\sqrt{1-4x}} - \frac{1}{2}$, so $\widetilde{f}_n = \frac{1}{2} \binom{2n}{n}$.
\item Let $\CC = \Av(312,231)$, the class of layered permutations, so called because each indecomposable block is a decreasing sequence. We have $F(x) = \frac{1-x}{1-2x}$ (the class is in bijection with compositions). Theorem \ref{thm:exactclass} yields $\widetilde{F}(x) = \frac{x}{(1-x)(1-2x)}$, so $\widetilde{f}_n = 2^n - 1$.
\item If $\CC = \Av(3142)$, then $F(x) = \frac{32x}{1+20x-8x^2-(1-8x)^{3/2}}$. Applying Theorem \ref{thm:exactclass} gives us $\widetilde{F}(x) = \frac{1-2x-\sqrt{1-8x}}{2(1+x)}$.
\item If $\CC = \Av(3142, 2413)$, the class of separable permutations, then $F(x) = \sfrac{3-x-\sqrt{1-6x+x^2}}{2}$ (this class is counted by the Schr\"oder numbers). Applying Theorem \ref{thm:exactclass}, we obtain $\widetilde{F}(x) = \frac{x}{\sqrt{1-6x+x^2}}$, so $\widetilde{f}_n$ is the $(n-1)$st central Delannoy number.
\end{itemize}
\end{expl}

We conclude this section with a result on the number of all decomposable affine permutations of size $n$.

\begin{prop} Let $\CC$ be the class of all permutations, so $\widetilde{f}_n = |\aff{S}_n|$ (the number of decomposable affine permutations of size $n$) and $g_n = |\ind{S}_n|$ (the number of indecomposable (ordinary) permutations of size $n$). For all $n \ge 1$, we have $\widetilde{f}_n = g_{n+1}$.
\end{prop}
\begin{proof}
Recall that every permutation in $S_{n+1}$ is obtained in a unique way by taking a permutation $\pi \in S_n$ and inserting the value $n+1$ in some position from $1$ to $n+1$. The permutation in $S_{n+1}$ is indecomposable if and only if the value $n+1$ is inserted in the first block of $\pi$, i.e.\ it is inserted in a position from $1$ to $a(\pi)$. Thus each $\pi \in S_n$ gives rise to $a(\pi)$ permutations in $\ind{S}_{n+1}$. Therefore, $g_{n+1} = \sum_{\pi \in S_n} a(\pi)$, which equals $\widetilde{f}_n$ by Lemma \ref{lem:pair}.
\end{proof}

\section{Supercritical and subcritical sequence schemas} \label{sec:schema}

A \emph{sequence schema} is a pair of ordinary generating functions $F(x)$ and $G(x)$ where $G(x)$ represents ``connected'' or ``indecomposable'' structures and $F(x)$ represents sequences of the connected structures. This is expressed by the relation
\[ F(x) = \frac{1}{1-G(x)}. \]
The constant term of $G(x)$ is $0$, and the constant term of $F(x)$ is $1$.

Let $r$ denote the radius of convergence of $G(x)$, and let $\tau = \lim_{x\to r^-} G(x)$.
\begin{itemize}
\item The sequence schema is \emph{subcritical} if $\tau < 1$.
\item The sequence schema is \emph{critical} if $\tau = 1$.
\item The sequence schema is \emph{supercritical} if $\tau > 1$.
\end{itemize}
The schema exhibits qualitatively different behavior depending on whether it is subcritical, critical, or supercritical. The reason it matters whether $\tau$ is greater than $1$ is that it affects the nature and location of the singularities of $F(x)$ (considered as a complex function). If $\tau > 1$ (the supercritical case), then by the Intermediate Value Theorem there is $\xi \in (0,r)$ such that $G(\xi) = 1$, which makes $\xi$ a singularity of $F(x) = \frac{1}{1-G(x)}$. Conversely, if $\tau > 1$ (the subcritical case), then this phenomenon does not occur, so there is no singularity on the interior of the disk of convergence --- the pertinent singularities of $F(x)$ are just the ones inherited from $G(x)$. The case of $\tau = 1$ (the critical case) is less common, and its analysis is more complicated, so we will mostly ignore it --- but see Example \ref{expl:critical} below.

If $\CC$ is a sum closed permutation class, with $F(x)$ and $G(x)$ the respective generating functions of $\CC$ and $\ind{\CC}$, then $F(x)$ and $G(x)$ are a sequence schema. Throughout this section and the next, we consider several examples of specific classes $\CC$.

\begin{expl}
   \label{ex.C}
\begin{itemize}
\item Let $\CC$ be one of $\Av(321)$, $\Av(312)$, or $\Av(231)$. We have $F(x) = \frac{1-\sqrt{1-4x}}{2x}$ and $G(x) = \frac{1-\sqrt{1-4x}}{2} = x\,F(x)$. We have $r = 1/4$, so
\[ \tau = \lim_{x\to1/4^-} \frac{1-\sqrt{1-4x}}{2} = \frac{1}{2}, \]
and so $\CC$ is subcritical.
\item Let $\CC = \Av(312,231)$, the class of layered permutations. We have $G(x) = \frac{x}{1-x}$ and $F(x) = \frac{1-x}{1-2x}$. Since $r = 1$, we get
\[ \tau = \lim_{x\to1^-} \frac{1}{1-x} = \infty, \]
and so this class is supercritical.
\item More subcritical classes: $\Av(3142)$, $\Av(3142, 2413)$ (i.e.\ the separable permutations), and any sum closed classes that are equinumerous to either of these. Another supercritical class: $\Av(321, 312, 231)$ (this is the class of permutations whose blocks are $1$ or $21$, and it is counted by the Fibonacci numbers).
\end{itemize}
\end{expl}
\begin{expl}
If $F(x)$ is rational (which is of course equivalent to $G(x)$ being rational), then the schema is supercritical. This is because the singularities of a rational function must be poles, and the limit at a pole is $\infty$. This covers both of the supercritical permutation classes given in the previous example.
\end{expl}

\begin{expl} \label{expl:critical}
We now construct a sum closed permutation class $\CC$ that has a critical sequence schema. Our method is to start with the indecomposable permutations in $\Av(321)$ (which is subcritical) and add more indecomposable permutations until the resulting schema is critical. For each $n$, define a set $A_n \subseteq \ind{S_n}$ as follows:
\begin{itemize}
\item If $1 \le n \le 7$, then let $A_n = \ind{S_n}$.
\item Let $A_8$ be any set of $6513$ permutations satisfying $\ind{\Av_8(321)} \subseteq A_8 \subseteq \ind{S_8}$. Such a set $A_8$ exists, since $|\ind{\Av_8(321)}| = 429$ and $|\ind{S_8}| = 29\,093$.
\item If $n \ge 9$, then let $A_n = \ind{\Av_n(321)}$.
\end{itemize}
Finally, define $\CC$ to be the set of permutations of the form $\pi_1 \oplus \cdots \oplus \pi_k$ for $\pi_1, \ldots, \pi_k \in \bigsqcup_{n\ge1} A_n$. This set $\CC$ is a permutation class because the sets $A_n$ satisfy the following property: if $\rho \in \ind{S_i}$ and $\pi \in A_j$ and $\rho$ is contained as a pattern in $\pi$, then $\rho \in A_i$. Clearly this class satisfies $\ind{\CC_n} = A_n$. Therefore,
\begin{align*}
G(x) &= \frac{1-\sqrt{1-4x}}{2} + \left[\sum_{i=1}^7 \left(|\ind{S_i}| - |\ind{\Av_i(321)}|\right)x^i\right] + \left(6513 - |\ind{\Av_8(321)}|\right)x^8 \\
&= \frac{1-\sqrt{1-4x}}{2} + x^3 + 8x^4 + 57x^5 + 419x^6 + 3315x^7 + 6084x^8.
\end{align*}
We have $\tau = G(1/4) = 1$ (this is the reason for our precise choice of $|A_8| = 6513$), so $\CC$ is critical. We are not aware of any  critical classes that are less artificial.
\end{expl}

\begin{rem} B\'ona \cite{Bona} showed that the class $\Av(\sigma)$ is critical or subcritical for almost all $\sigma \in \ind{S_k}$ (in the sense that the fraction of $\sigma \in \ind{S_k}$ for which the property holds goes to $1$ as $k \to \infty$). It is not hard to strengthen this result to show that the class $\Av(\sigma)$ is in fact subcritical for almost all $\sigma \in \ind{S_k}$. \end{rem}

Let $\chi_n$ be the number of blocks (or components) of a uniformly random size-$n$ structure counted by $F(x)$. We are interested in the asymptotic distribution of $\chi_n$ as $n \to \infty$. 
We now state the relevant results, from Flajolet and Sedgewick \cite{FS}.  

\begin{prop}[{\cite[Prop.\ IX.2]{FS}}]
\label{prop:subcritical} Consider the subcritical case $\tau<1$.  Assume that when extended to a function of a complex variable $z$, $G$ satisfies the following technical 
conditions:
$G(z)$ has no singularity except $z=r$ on its disc of convergence, and we can express $G(z)=\tau-(c+o(1))(1-z/r)^{\lambda}$ with $c>0$ and 
$0<\lambda<1$ for $z$ in some set of the form $D\setminus T$ where $D=\{z\in\mathbb{C}:|z|<r+\epsilon\}$ ($\epsilon>0$)
and $T$ is a
closed triangle, symmetric about the real axis, with its left vertex at $z=r$.
Then, as $n\to\infty$, the mean of $\chi_n$ is asymptotically $\frac{1+\tau}{1-\tau}$, and the variance is asymptotically $\frac{2\tau}{(1-\tau)^2}$. Moreover, the distribution of $\chi_n$ converges to a negative binomial distribution (offset by $1$):
\[ \lim_{n\to\infty} \mathbb{P}[\chi_n = k] = (1-\tau)^2 k \tau^{k-1} \]
for positive integer $k$. In particular, setting $k=1$ yields $\lim_{n\to\infty} g_n/f_n = (1-\tau)^2$.
\end{prop}

\begin{prop}[{\cite[Thm.\ V.1 \& Prop.\ IX.7]{FS}}]
\label{prop:supercritical} In the supercritical case, let $\rho$ be the (unique) positive root of $G(x) = 1$, which is also the radius of convergence of $F(x)$. Assume that the coefficients of $G(x)$ are aperiodic (in the sense that there do not exist integers $r \ge 0$ and $d \ge 2$ and a power series $H(x)$ such that $G(x) = x^r H(x^d)$). Define
\[ \alpha = \frac{1}{\rho G'(\rho)} \quad \text{and} \quad \beta = \frac{\rho G''(\rho) + G'(\rho) - \rho G'(\rho)^2}{\rho^2 G'(\rho)^3}. \]
Then:
\begin{enumerate}[(a)]
\item The mean of $\chi_n$ is asymptotically $\alpha n$, and the variance is asymptotically $\beta n$. (Note that the variance is given incorrectly in \cite[Prop.\ IX.7]{FS}, but correctly in \cite[Thm.\ V.1]{FS}.)
\item The distribution of $\chi_n$ is concentrated, in the sense that, for every $\varepsilon > 0$, the probability that $|\chi_n - \alpha n| < \varepsilon n$ converges to $1$. (This follows directly from (a).)
\item The distribution of the standardized variable $\frac{\chi_n - \alpha n}{\sqrt{\beta n}}$ converges to the standard normal distribution:
\[ \lim_{n\to\infty} \mathbb{P}\left[\chi_n < \alpha n + t\sqrt{\beta n}\right] = \int_{-\infty}^t \frac{1}{\sqrt{2\pi}} e^{-u^2/2}\,du \]
for $t \in \mathbb{R}$.
\item The coefficient of $x^n$ in $F(x)$ is asymptotically $\alpha \rho^{-n}$.
\end{enumerate}
\end{prop}

The treatment of asymptotics for the critical case \cite[Subsec.\ IX.11.2]{FS} is more subtle, and we do not cover it.

\begin{expl} \label{expl:asympt1}
\begin{itemize}
\item Let $\CC$ be one of $\Av(321)$, $\Av(312)$, or $\Av(231)$.  The function $G$ is given in Example \ref{ex.C}
above.  Since $\CC$ has $\tau = 1/2$, Proposition \ref{prop:subcritical} says that the expected number of blocks of a uniformly random permutation in $\CC_n$ is asymptotically $3$, and the number of indecomposable permutations, $g_n$, satisfies $g_n \sim \sfrac{1}{4} f_n$. The probability that a uniformly random permutation in $\CC_n$ has $k$ blocks is asymptotically $\sfrac{1}{4} k \left(\sfrac{1}{2}\right)^{k-1}$.
\item Let $\CC = \Av(3142)$. The function $G(x)$ is given by
Equation (4.10) of \cite{BonaCP} (with Lemma 4.13 of
\cite{BonaCP}).
Since $\CC$ has $\tau = 5/32$, Proposition \ref{prop:subcritical} says that the expected number of blocks is asymptotically $37/27 \approx 1.37$, and that $g_n \sim \sfrac{729}{1024} f_n$. The probability that a uniformly random permutation has $k$ blocks is asymptotically $\sfrac{729}{1024} k \left(\sfrac{5}{32}\right)^{k-1}$.
\item Let $\CC = \Av(3142, 2413)$ (separable permutations). We have $\tau = 1 - \sfrac{1}{\sqrt{2}} \approx 0.293$; the expected number of blocks is asymptotically $2\sqrt{2}-1 \approx 1.83$; we have $g_n \sim \sfrac{1}{2} f_n$ (in this case it happens that $g_n = \sfrac{1}{2} f_n$); and the probability of $k$ blocks is asymptotically $\sfrac{1}{2} k \left(1-\sfrac{1}{\sqrt{2}}\right)^{k-1}$.
\item Let $\CC = \Av(312, 231)$ (layered permutations). We have $G(x) = \frac{x}{1-x}$, from which we compute $\rho = 1/2$ and $\alpha = 1/2$ and $\beta = 1/4$, Thus, by Proposition \ref{prop:supercritical}, the expected value of the number of blocks is asymptotically $\sfrac{1}{2} n$, and the variance is asymptotically $\sfrac{1}{4} n$. After standardization (using this mean and variance), the number of blocks has a normal limiting distribution. We also conclude that $f_n \sim \sfrac{1}{2} \cdot 2^n$. Of course, since these permutations are in bijection with compositions, the distribution of the number of blocks is a binomial distribution, and in fact $f_n = \sfrac{1}{2} \cdot 2^n$.
\item Let $\CC = \Av(321, 312, 231)$ (permutations whose blocks are all $1$ or $21$, counted by the Fibonacci numbers). In this case we simply have $G(x) = x + x^2$, from which we compute $\rho = \sfrac{\sqrt{5}-1}{2} \approx 0.618$, $\alpha = \sfrac{2}{(\sqrt{5}-1)\sqrt{5}} \approx 0.724$ and $\beta = \sfrac{1}{5\sqrt{5}} \approx 0.0894$. By Proposition \ref{prop:supercritical}, the expected number of blocks is asymptotically $\sfrac{2}{(\sqrt{5}-1)\sqrt{5}} n$ with variance $\sfrac{1}{5\sqrt{5}} n$. After standardization, the number of blocks has a normal limiting distribution. Finally, $f_n \sim \sfrac{2}{(\sqrt{5}-1)\sqrt{5}} \cdot \left(\sfrac{\sqrt{5}+1}{2}\right)^n$ (which we also know because $f_n$ is the $(n+1)$st Fibonacci number).
\end{itemize}
\end{expl}

We conclude this section with a refined probabilistic 
description of random elements in the subcritical case 
described by Proposition \ref{prop:subcritical}.  
Under the assumptions of Proposition \ref{prop:subcritical}, a random element of size $n$ is very likely to have one very large block along with a few
small blocks on each side of the large block.

Recall that $a(\pi)$ is the size of the first block of the permutation $\pi$.  We shall write $a[n]$ to denote the
random variable $a(\pi)$ when $\pi$ is chosen uniformly at
random from the elements of size $n$.  In the following result, the main interest is in the case that $b_n$ tends 
to infinity very slowly.

\begin{prop}
   \label{prop.bigblock}
Assume that the hypotheses of Proposition \ref{prop:subcritical} hold.  Let $\{b_n\}$ be a sequence 
of natural numbers that tend to infinity such that 
$b_n=o(n)$.  Let $\mathcal{A}_n$ be the event that for
a randomly chosen element of size $n$, one 
block has size at least $n-b_n$ and no other block has size 
greater than $b_n$.
Then 
\\
(a) $\lim_{n\rightarrow\infty}\mathbb{P}(a[n]=j) \,=\, g_jr^j$ for all $j\in\mathbb{N}$.
\\
(b)  $\lim_{n\rightarrow\infty}\mathbb{P}(a[n]\leq b_n)=\tau$ and 
$\lim_{n\rightarrow\infty}\mathbb{P}(a[n]\geq n-b_n) \,=\, 1-\tau$.
\\
(c)  Fix integers $k\geq j\geq 1$.  Then, conditional on the
event that $\chi_n=k$, the probability that the size of the $j^{th}$ block lies in the interval $(b_n,n-b_n)$  converges to 0 as $n\rightarrow\infty$.
\\
(d)
$\lim_{n\to\infty}\mathbb{P}(\mathcal{A}_n|\chi_n=k)\,=\,1$
for every $k\in\mathbb{N}$.
\\
(e) $\lim_{n\to\infty}\mathbb{P}(\mathcal{A}_n)\,=\,1$. 
\end{prop}

\begin{proof}
Note that the assumptions imply that $\lim_{n\rightarrow\infty} f_n/f_{n+1}=r$ (see equation (15) in the proof of Proposition IX.1 of \cite{FS}, whose hypotheses are the same as those of Proposition \ref{prop:subcritical}).
Part (a) follows from this and the identity $\mathbb{P}(a[n]=j)\,=\, g_jf_{n-j}/f_n$.

For part (b), define the sequences $u_n$, $v_n$ and $w_n$ by
\[   u_n\,=\,\sum_{j=1}^{b_n}g_jf_{n-j}\,, 
  \hspace{5mm} v_n\,=\,\sum_{j=b_n+1}^{n-b_n-1}g_jf_{n-j}\,,
   \hspace{5mm} w_n\,=\,\sum_{j=n-b_n}^{n}g_jf_{n-j}
     \,=\,\sum_{i=0}^{b_n}g_{n-i}f_{i}\,.
\]
Observe that $u_n+v_n+w_n=f_n$.  By part (a) we have
$\liminf_{n\rightarrow\infty}u_n/f_n\geq G(r) = \tau$.  
Recall from Proposition \ref{prop:subcritical}
that $\lim_{n\rightarrow\infty}g_n/f_n=(1-\tau)^2$.
And since $\lim_{n\rightarrow\infty}f_n/f_{n+1}=r$, 
it follows that $\lim_{n\rightarrow\infty}g_n/g_{n+1}=r$.
Moreover,
\[    \lim_{n\rightarrow\infty} \frac{g_{n-i}}{f_n} \;=\;
    \lim_{n\rightarrow\infty} \frac{g_{n-i}}{g_n}\frac{g_n}{f_n} \;=\; r^i(1-\tau)^2,
\]
and hence 
\[  \liminf_{n\rightarrow\infty}\frac{w_n}{f_n} \;\geq\; 
    \sum_{i=0}^{\infty} r^if_i(1-\tau)^2   \;=\; F(r)(1-\tau)^2   \;=\; 1-\tau
\]
because $F(r)=1/(1-\tau)$.  Since $\frac{u_n}{f_n}+\frac{v_n}{f_n}+\frac{w_n}{f_n}=1$, 
we must have
\[    \lim_{n\rightarrow\infty} \frac{u_n}{f_n}\;=\;\tau\,,
   \hspace{5mm} 
   \lim_{n\rightarrow\infty} \frac{v_n}{f_n}\;=\;0\,,
    \hspace{5mm} \hbox{and}\hspace{5mm}
   \lim_{n\rightarrow\infty} \frac{w_n}{f_n}\;=\;1-\tau\,.
\]
The first and third limits prove part (b).  The second limit
proves that 
$\lim_{n\rightarrow\infty}\mathbb{P}(b_n<a[n]<n-b_n)=0$.
Now fix $k\in\mathbb{N}$.  Since $\lim_{n\rightarrow\infty}\mathbb{P}(\chi_n=k)$ is nonzero, 
it follows that
\begin{equation}
    \label{eq.blockbig}
  \limsup_{n\rightarrow\infty} \mathbb{P}(b_n<a[n]<n-b_n\,|\,\chi_n=k)  \;\leq \;
     \limsup_{n\rightarrow\infty}
    \frac{\mathbb{P}(b_n<a[n]<n-b_n)}{\mathbb{P}(\chi_n=k)}
    \;=\; 0.    
\end{equation}
For fixed $\chi_n=k$, the distribution of the size of the 
$j^{th}$ block is the same for every $j$.  This observation 
and equation (\ref{eq.blockbig}) imply part (c).
Part (d) is a consequence of part (c), since some block has size at least $n/k$, and at most one block can have size greater than or equal to $n-b_n$ 
(recall $b_n=o(n)$, so  $b_n<\frac{n}{k}$ and 
$\frac{n}{2} <n-b_n$ for large $n$).
Finally, part (e) follows from part (d) and the fact that
$\chi_n$ converges in distribution (Proposition \ref{prop:subcritical}).
\end{proof}

\section{Asymptotic enumeration of decomposable affine permutation classes} \label{sec:asympt}

We continue to assume that $\CC$ is a sum closed permutation class, letting $F(x)$, $\widetilde{F}(x)$, and $G(x)$ denote the respective generating functions of $\CC$, $\aff{\CC}$, and $\ind{\CC}$. Also recall that $\first(\pi)$ is the size of the first indecomposable block of $\pi$, and that $\chi(\pi)$ is the number of indecomposable blocks of $\pi$.

We begin by proving a lemma that we will need for both asymptotic results in this section:

\begin{lem} \label{lem:nsum}
For $n \ge 1$, $\widetilde{f}_n = |\aff{\CC}_n| = n \sum_{\pi \in \CC_n} \frac{1}{\chi(\pi)}$.
\end{lem}
\begin{proof} \let\qqed\qed \let\qed\negativespace 
For permutations $\pi, \pi' \in S_n$, say that $\pi \approx \pi'$ if there are permutations $\sigma$ and $\tau$ such that $\pi = \sigma \oplus \tau$ and $\pi' = \tau \oplus \sigma$. Let $[\pi]$ denote the equivalence class of $\pi$ under $\approx$; that is, $[\pi]$ is the set of permutations obtained by cyclically permuting the summands in the sum decomposition of $\pi$. Since $\CC$ is sum closed, it is closed under the equivalence relation $\approx$, in the sense that if $\pi \in \CC$ then $[\pi] \subseteq \CC$. Let $\CC_n/{\approx}$ denote the set of equivalence classes of $\CC_n$ under $\approx$.

For $\pi \in S_n$, define the \emph{period} of $\pi$, denoted $\lambda(\pi)$, to be the smallest $j$ such that $\pi$ is a sum of $n/j$ copies of a permutation of size $j$:
\[ \pi = \bigoplus_{i=1}^{n/j} \rho \]
for some $\rho \in S_j$. Observe that the quantities we have defined are related by
\begin{equation} \label{eq:relation}
|[\pi]| = \frac{\chi(\pi)\,\lambda(\pi)}{n}
\end{equation}
and
\begin{equation} \label{eq:lambda}
\sum_{\pi' \in [\pi]} a(\pi') = \lambda(\pi).
\end{equation}
Now,
\begin{align*}
\widetilde{f}_n &= \sum_{\pi \in \CC_n} \first(\pi) & \text{(by Lemma \ref{lem:pair})} \\
&= \sum_{[\pi] \in \CC_n/{\approx}} \,\sum_{\pi' \in [\pi]} a(\pi') \\
&= \sum_{[\pi] \in \CC_n/{\approx}} \lambda(\pi) & \text{(by \eqref{eq:lambda})} \\
&= \sum_{\pi \in \CC_n} \frac{\lambda(\pi)}{|[\pi]|} \\
&= \sum_{\pi \in \CC_n} \frac{n}{\chi(\pi)}. & \text{(by \eqref{eq:relation})} \tag*{\qqed}
\end{align*}
\end{proof}

The following elementary lemma will be used to prove Theorem \ref{prop:subasympt}.  It can be viewed as a property 
of convergence in distribution for discrete distributions.
\begin{lem}
   \label{lem.limE}
Let $p(\cdot),p_1(\cdot),p_2(\cdot),\ldots$ be probability 
distributions on $\mathbb{N}$ such that $\lim_{n\rightarrow\infty}p_n(k)=p(k)$ for every $k\in\mathbb{N}$.  Then for every bounded sequence $\{a_k\}$,
\[    \lim_{n\rightarrow\infty}\sum_{k=1}^{\infty}
    a_k\, p_n(k)  \;=\;  \sum_{k=1}^{\infty} a_k\,p(k).
\]
\end{lem}

\begin{thm} \label{prop:subasympt}
If $\CC$ is subcritical and its corresponding generating functions $F(x)$ and $G(x)$ satisfy the complex-analytic conditions alluded to in Proposition \ref{prop:subcritical}, then $\widetilde{f}_n \sim (1-\tau)n f_n$ and $g_n \sim (1-\tau)^2 f_n$, with $\tau$ defined as in Section \ref{sec:schema}.
\end{thm}
\begin{proof} \let\qqed\qed \let\qed\negativespace 
The statement $g_n \sim (1-\tau)^2 f_n$ is covered by Proposition \ref{prop:subcritical}. We prove the other statement:
\begin{align*}
\widetilde{f}_n &= n \sum_{\pi \in \CC_n} \frac{1}{\chi(\pi)} & \text{(by Lemma \ref{lem:nsum})} \\
&= n \sum_{k=1}^n \sum_{\substack{\pi\in\CC_n\\ \chi(\pi) = k}} \frac{1}{k} \\
&= n \sum_{k=1}^n \frac{1}{k} \cdot \#\{\pi \in \CC_n \colon \chi(\pi) = k\} \\
&= n \sum_{k=1}^n \frac{1}{k} \,f_n \,\mathbb{P}[\chi_n = k] \\
&\sim n f_n \sum_{k=1}^n \frac{1}{k} (1-\tau)^2 k \tau^{k-1} & \text{(by Proposition \ref{prop:subcritical} and Lemma \ref{lem.limE})} \\
&\sim (1-\tau)^2 n f_n \sum_{k=1}^\infty \tau^{k-1} \\
&= (1-\tau) n f_n. \tag*{\qqed}
\end{align*}
\end{proof}

\begin{thm} \label{prop:superasympt}
If $F(x)$ and $G(x)$ are a supercritical sequence schema and $G(x)$ satisfies the aperiodicity condition from Proposition \ref{prop:supercritical}, then
\[ \widetilde{f}_n \sim \alpha^{-1} f_n \sim \rho^{-n}, \]
with $\rho$ and $\alpha$ defined as in Proposition \ref{prop:supercritical}: $\rho$ is the (unique) positive root of $G(x) = 1$, which is also the radius of convergence of $F(x)$, and $\alpha^{-1} = \rho G'(\rho)$.
\end{thm}
\begin{proof} \let\qqed\qed \let\qed\negativespace 
\begin{align*}
\widetilde{f}_n &= \sum_{\pi \in \CC_n} \frac{n}{\chi(\pi)} & \text{(by Lemma \ref{lem:nsum})} \\
&\sim \sum_{\pi \in \CC_n} \frac{1}{\alpha} & \text{(by Proposition \ref{prop:supercritical}(b))} \\
&= \alpha^{-1} f_n \\
&\sim \rho^{-n} & \text{(by Proposition \ref{prop:supercritical}(d))} \tag*{\qqed}
\end{align*}
\end{proof}

\begin{rem} We can interpret these results in the context of Proposition \ref{prop:exact}(b), which says that
\begin{equation} \label{eq:bounds} \max\{ng_n, f_n\} \le \widetilde{f}_n \le n f_n.\end{equation}
\begin{itemize}
\item In the subcritical case, $g_n \sim (1-\tau)^2 f_n$, so \eqref{eq:bounds} becomes
\[ (1-\tau)^2 nf_n \lesssim \widetilde{f}_n \le n f_n. \]
We now see that $\widetilde{f}_n$ falls at the geometric mean of these bounds, since by Theorem \ref{prop:subasympt} we have $\widetilde{f}_n \sim (1-\tau)nf_n$.
\item In the supercritical case, $g_n$ is exponentially smaller than $f_n$, so \eqref{eq:bounds} becomes
\[ f_n \le \widetilde{f}_n \le n f_n. \]
We now see by Theorem \ref{prop:superasympt} that $\widetilde{f}_n$ falls near the lower bound, since asymptotically it is a constant multiple of $f_n$ and not of $n f_n$.
\end{itemize} \end{rem}

\begin{expl}
Applying Theorems \ref{prop:subasympt} and \ref{prop:superasympt} to the permutation classes from Example \ref{expl:asympt1} yields:
\begin{itemize}
\item If $\CC = \Av(321)$, then $g_n \sim \sfrac{1}{4} f_n$ and $\widetilde{f}_n \sim \sfrac{1}{2} n f_n$. Since $f_n \sim \sfrac{1}{\sqrt{\pi}} n^{-3/2} 4^n$, we obtain $g_n \sim \sfrac{1}{4\sqrt{\pi}} n^{-3/2} 4^n$ and $\widetilde{f}_n \sim \sfrac{1}{2\sqrt{\pi}} n^{-1/2} 4^n$.
\item If $\CC = \Av(3142)$, then $g_n \sim \sfrac{729}{1024} f_n$ and $\widetilde{f}_n \sim \sfrac{27}{32} n f_n$. 
\item If $\CC = \Av(3142, 2413)$ (separable permutations), then $g_n \sim \sfrac{1}{2} f_n$ and $\widetilde{f}_n \sim \sfrac{1}{\sqrt{2}} n f_n$. 
\item If $\CC = \Av(312, 231)$ (layered permutations), then $\widetilde{f}_n \sim 2f_n \sim 2^n$.
\item If $\CC = \Av(321, 312, 231)$ (permutations whose blocks are all $1$ or $21$), then $\widetilde{f}_n \sim \sfrac{(\sqrt{5}-1)\sqrt{5}}{2} f_n \sim \left(\sfrac{\sqrt{5}+1}{2}\right)^n$.
\end{itemize}
\end{expl}

\section{Recognizing when an affine permutation class is decomposable} \label{sec:recognizing}

A class of affine permutations may have the property that every element of the class is decomposable. We will say that a class with this property is \emph{decomposable}. Note that, for a set of (ordinary) permutations $R$, the affine class $\AvA(R)$ is decomposable if and only if $\AvA(R) = \oplus \Av(R)$, if and only if $\AvA(R) = \AvBA(R) = \oplus \Av(R)$. Thus, if $\AvA(R)$ is decomposable, then the methods of the previous sections give us exact and asymptotic enumeration not only of $\oplus \Av(R)$ but of $\AvBA(R)$ and $\AvA(R)$.

There is an easy way to tell whether a class is decomposable, involving an affine permutation called the \emph{infinite (increasing) oscillation}. This is the permutation
\[ \OO = \begin{pmatrix}\cdots & -1 & 0 & 1 & 2 & 3 & 4 & \cdots \\ \cdots & 1 & -2 & 3 & 0 & 5 & 2 & \cdots\end{pmatrix} \in \widetilde{S}_2, \]
as shown in Figure \ref{fig:oscillation}.
\begin{figure}
\begin{center}
\begin{tikzpicture}[scale=0.25]
\draw (1,3) [fill=black] circle (.4);
\draw (2,0) [fill=black] circle (.4);
\draw (3,5) [fill=black] circle (.4);
\draw (4,2) [fill=black] circle (.4);
\draw (5,7) [fill=black] circle (.4);
\draw (6,4) [fill=black] circle (.4);
\draw (7,9) [fill=black] circle (.4);
\draw (8,6) [fill=black] circle (.4);
\draw (9,11) [fill=black] circle (.4);
\draw (10,8) [fill=black] circle (.4);
\draw (11,13) [fill=black] circle (.4);
\draw (12,10) [fill=black] circle (.4);
\draw [thick] (-2,4) -- (15,4);
\draw [thick] (4,-2) -- (4,15);
\node at (-1,-0.5) {$\iddots$};
\node at (14,14) {$\iddots$};
\draw [very thick] (1,3) -- (2,0);
\draw [very thick] (1,3) -- (4,2);
\draw [very thick] (3,5) -- (4,2);
\draw [very thick] (3,5) -- (6,4);
\draw [very thick] (5,7) -- (6,4);
\draw [very thick] (5,7) -- (8,6);
\draw [very thick] (7,9) -- (8,6);
\draw [very thick] (7,9) -- (10,8);
\draw [very thick] (9,11) -- (10,8);
\draw [very thick] (9,11) -- (12,10);
\draw [very thick] (11,13) -- (12,10);
\draw [thick] (-2,4) -- (15,4);
\draw [thick] (4,-2) -- (4,15);
\end{tikzpicture}
\end{center}
\caption{The infinite oscillation $\OO \in \widetilde{S}_2$. The edges indicate inversions of $\OO$ and form the inversion graph $G(\OO)$.}
\label{fig:oscillation}
\end{figure}
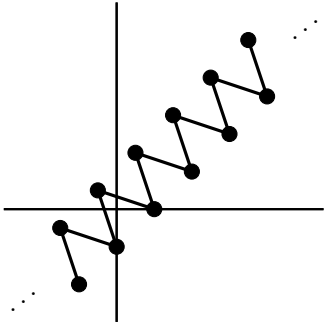
We take $\OO$ to be of size $2$. The \emph{finite (increasing) oscillations} are the indecomposable ordinary permutations that are contained in $\OO$. For each $n \ge 3$, there are exactly two of size $n$. The finite oscillations are:
\[ 1, 21, 312, 231, 3142, 2413, 31524, 24153, 315264, 241635, \ldots \]
The infinite oscillation and the finite oscillations have made appearances in research on antichains in the permutation containment order \cite{ABV} and on growth rates of permutation classes \cite{Vatter}. Here they arise as the chief obstruction to decomposability.

\begin{thm} \label{thm:oscillation}
\begin{enumerate}[(a)]
\item An affine permutation is decomposable if and only if it avoids $\OO$.
\item An affine permutation class is decomposable, if and only if it does not have $\OO$ as an element, if and only if it is a subset of $\AvA(\OO)$.
\item Let $R$ be a set of (ordinary) permutations; then $\oplus \Av(R) = \AvA(R \cup \{\OO\})$. In particular, $\AvA(R)$ is decomposable if and only if there is $\tau \in R$ that is contained in $\OO$.
\item Let $R$ be a set of indecomposable (ordinary) permutations; $\AvA(R)$ is decomposable if and only if $R$ has an element that is a finite oscillation.
\end{enumerate}
\end{thm}

Note that with parts (c) and (d) we can easily find, for every set $R$ of permutations, whether $\AvA(R)$ is decomposable.

\begin{rem}
This theorem is similar to the result by Crites \cite{Crites} that $\AvA(\tau)$ satisfies $|\AvA_n(\tau)| < \infty$ for all $n$ if and only if $\tau$ avoids $321$ --- or, more generally, an affine permutation class $\CC$ satisfies $|\CC_n| < \infty$ for all $n$ if and only if $\AvA(321)$ is not a subset of $\CC$. This makes possible the intriguing phenomenon of an affine permutation class $\CC$, such as $\AvA(3412)$, that is not decomposable but has $|\CC_n| < \infty$ for each $n$. Theorem \ref{thm:oscillation} also implies that an affine permutation class is a subset of $\E$ if and only if it is decomposable: one direction was already clear without our theorem, and the other direction holds because an affine permutation class that is not decomposable must contain $\OO$, which is not bounded (considering $\OO$ to have size $2$, see Remarks \ref{rem:disjoint1} and \ref{rem:disjoint2}).
\end{rem}

Before proving Theorem \ref{thm:oscillation}, we will introduce the \emph{inversion graph} of a permutation, and we will give a few required lemmas.

Let $A$ and $B$ be ordered sets, and let $\omega\colon A \to B$ be a bijection. An \emph{inversion} of $\omega$ is a pair $\{i,j\} \subseteq A$ such that $i < j$ and $\omega(i) > \omega(j)$. The \emph{inversion graph} of $\omega$ is the graph $G(\omega)$ with vertex set $A$ whose edges are the inversions of $\omega$. Abusing notation, we will sometimes say ``$\omega$ has property $P$'' to mean ``$G(\omega)$ has property $P$''.

The following lemma is evident:
\begin{lem} \label{lem:evident}
An affine permutation is decomposable if and only if its inversion graph is not connected. More specifically, the blocks of an affine permutation's sum decomposition are the components of its inversion graph. \hfill $\square$
\end{lem}
(This lemma's analog for ordinary permutations is well known, and the proof is the same.) For example, we see in Figure \ref{fig:oscillation} that $G(\OO)$ is connected --- it is a doubly infinite path --- so the lemma tells us that $\OO$ is not decomposable.

\begin{lem} \label{lem:oshift}
Let $\omega$ be a bijection from $\Z$ to $\Z$. If $G(\omega)$ is a doubly infinite path, then $\omega$ is order-isomorphic to either $\OO$ or $\Sigma \OO$. That is, the only permutations of $\Z$ that are doubly infinite paths are order-isomorphic to one of the two shifts of $\OO$.
\end{lem}
\begin{proof}
Assume $G(\omega)$ is a doubly infinite path. Every vertex of $G(\omega)$ has exactly two neighbors. Let $j \in \Z$, and let $i,k$ be the neighbors of $j$ in $G(\omega)$. Suppose $i < j < k$. Since $\{i,j\}$ and $\{j,k\}$ are inversions of $\omega$, we have $\omega(i) > \omega(j) > \omega(k)$. But this means $\{i,k\}$ is an inversion of $\omega$, so $i,j,k$ induce a triangle in $G(\omega)$, a contradiction. Therefore, the two neighbors of $j$ must be ``on the same side'' of $j$: either $i,k < j$ and $\omega(i),\omega(k) > \omega(j)$, or $i,k > j$ and $\omega(i),\omega(k) < \omega(j)$.

Let
\[ \cdots \to x_0 \to x_1 \to x_2 \to \cdots \]
denote the vertices in the order in which they occur along the path. Without loss of generality, $x_1 = 1$. By the previous paragraph, either $x_0, x_2 > x_1$ or $x_0,x_2 < x_1$. First, assume that $x_0, x_2 > x_1$. Further assume that $x_0 < x_2$. We will prove that, under these assumptions ($x_1 < x_0 < x_2$), the relative order of the values of $\omega$ is uniquely determined.

Since $\{x_1, x_2\}$ is an inversion, we have $\omega(x_1) > \omega(x_2)$. Since $\{x_0, x_2\}$ is not an inversion, we have $\omega(x_0) < \omega(x_2)$. Thus,
\[ \omega(x_0) < \omega(x_2) < \omega(x_1). \]
That is, the three vertices $x_0, x_1, x_2$ form a permutation that is order-isomorphic to $312$.

The two neighbors of $x_2$ are $x_1$ and $x_3$. Since $x_1 < x_2$, we must also have $x_3 < x_2$ (by the ``same side'' principle). Since $\{x_2,x_3\}$ is an inversion, we have $\omega(x_3) > \omega(x_2)$. Since $\{x_0,x_3\}$ is not an inversion and $\omega(x_0) < \omega(x_2) < \omega(x_3)$, we have $x_0 < x_3$. Since $\{x_1, x_3\}$ is not an inversion and $x_1 < x_0 < x_3$, we have $\omega(x_1) < \omega(x_3)$. Therefore,
\[ x_1 < x_0 < x_3 < x_2 \quad \text{and} \quad \omega(x_0) < \omega(x_2) < \omega(x_1) < \omega(x_3). \]
That is, the four vertices $x_0, x_1, x_2, x_3$ form a permutation that is order-isomorphic to $3142$.

Continuing in this fashion, it follows that, for each $n \ge 4$, the position $x_n$ is uniquely determined relative to $x_0, \ldots, x_{n-1}$, and the value $\omega(x_n)$ is uniquely determined relative to $\omega(x_0), \ldots, \omega(x_{n-1})$. The same process for $x < 0$ completes the proof that the entire permutation $\omega$ is uniquely determined up to order isomorphism, assuming the condition $x_1 < x_0 < x_2$. But $\OO$ also satisfies this condition (setting $x_1 = 1$) along with the condition that $G(\OO)$ is a doubly infinite path, so it must be that $\omega$ is order-isomorphic to $\OO$.

The case where $x_1 < x_2 < x_0$ also results in $\omega$ being order-isomorphic to $\OO$, as exchanging the positions of $x_0$ and $x_2$ is tantamount to traversing the path in the opposite direction from $x_1 = 1$. In the case where $x_0, x_2 < x_1$, a similar argument shows that $\omega$ is order-isomorphic to $\Sigma \OO$.
\end{proof}

Recall that, if $X$ is a subset of the vertices of $G$, then the subgraph of $G$ \emph{induced by $X$} is the subgraph of $G$ whose vertex set is $X$ and whose edge set is the set of all edges of $G$ that join vertices in $X$. Furthermore, a subgraph $H$ of $G$ is an \emph{induced subgraph} of $G$ if $H$ is the subgraph of $G$ induced by the vertices of $H$; that is, every edge of $G$ that joins vertices of $H$ is also an edge of $H$.
\begin{lem} \label{lem:induced}
If an affine permutation is connected (i.e.\ not decomposable), then it has an induced subgraph that is a doubly infinite path.
\end{lem}

\begin{proof} Let $\omega$ be an affine permutation of size $n$, and assume that $G(\omega)$ is connected. Thus there is a path from $0$ to $n$ in $G(\omega)$. Denote the vertices in the path by
\[ 0 = x_0 \to x_1 \to \cdots \to x_{m-1} \to x_m = n. \]
We can continue this path infinitely in both directions by setting $x_{r+qm} = x_r + qn$ for $q\in\Z$ and $r \in [0,m-1]$, which results in the following doubly infinite walk:
\begin{align*}
\cdots \to (x_0 - n) \to (x_1 - n) \to &\cdots \to (x_{m-1}-n) \to \\
x_0 \to x_1 \to &\cdots \to x_{m-1} \to \\
(x_0 + n) \to (x_1 + n) \to &\cdots \to (x_{m-1} + n) \to \\
(x_0 + 2n) \to (x_1 + 2n) \to &\cdots \to (x_{m-1}+2n) \to \cdots
\end{align*}
For the rest of the proof, our goal is to find a doubly infinite subsequence ${\left(x_{\alpha(i)}\right)}_{i\in\Z}$ of this walk whose vertices induce a subgraph of $G(\omega)$ that is a path.

Set $\alpha(1) = 0$, and for each $i \ge 2$ define $\alpha(i)$ recursively as the greatest integer $a$ such that $x_a$ is adjacent to $x_{\alpha(i-1)}$ in $G(\omega)$ (there is a greatest such $a$ because $x_{\alpha(i-1)}$ is in finitely many inversions). Thus we have defined $\alpha(i)$ for each $i \ge 1$. Our construction ensures that the sequence $x_{\alpha(1)} \to x_{\alpha(2)} \to \cdots$ is a (singly) infinite path in which no two non-consecutive vertices are adjacent in $G(\omega)$ --- that is, if $x_{\alpha(i)}$ and $x_{\alpha(i')}$ are adjacent in $G(\omega)$ then $i' = i \pm 1$. From this path, we will construct a doubly infinite path with the same property.

Let $\alpha(0)$ be the least $a \in \Z$ such that $x_a$ is adjacent to $x_{\alpha(j)}$ for some $j \ge 1$. Now, for each $i \le -1$, define $\alpha(i)$ recursively as the least integer $a$ such that $x_a$ is adjacent to $x_{\alpha(i+1)}$ in $G(\omega)$. Thus we have defined $\alpha(i)$ for each $i \in \Z$. Finally, let $I$ be the greatest $i$ such that $x_{\alpha(i)}$ is adjacent to $x_{\alpha(0)}$. This ensures that the sequence
\[ \cdots \to x_{\alpha(-2)} \to x_{\alpha(-1)} \to x_{\alpha(0)} \to x_{\alpha(I)} \to x_{\alpha(I+1)} \to x_{\alpha(I+2)} \to \cdots \]
(skipping $x_{\alpha(1)}$ through $x_{\alpha(I-1)}$) is a doubly infinite path in which no two non-consecutive vertices are adjacent in $G(\omega)$. Indeed, consider vertices $x_{\alpha(i)}$ and $x_{\alpha(i')}$ with $i' \ge i+2$. We already saw that $x_{\alpha(i)}$ and $x_{\alpha(i')}$ are not adjacent when $i, i' \ge I$, and the same reasoning shows it when $i, i' \le 0$. If $i \le -1$ and $i' \ge I$, then our choice of $\alpha(0)$ implies that $x_{\alpha(i)}$ and $x_{\alpha(i')}$ are not adjacent. If $i = 0$ and $i' \ge I+1$, then our choice of $I$ implies that $x_{\alpha(i)}$ and $x_{\alpha(i')}$ are not adjacent.
\end{proof}

\begin{proof}[of Theorem \ref{thm:oscillation}] \emph{(a)} If an affine permutation $\omega$ is decomposable, then every block of $\omega$ has finite size, which by Lemma \ref{lem:evident} means that every component of $G(\omega)$ has finite size. But if $\OO$ is contained in $\omega$, then the entries of $\omega$ in an occurrence of $\OO$ form an infinite connected subgraph of $G(\omega)$, a contradiction. Therefore, if $\omega$ is decomposable, then $\omega$ avoids $\OO$.

Now assume $\omega$ is not decomposable. By Lemma \ref{lem:induced}, $G(\omega)$ has an induced subgraph $H$ that is a doubly infinite path. We can delete all entries of $\omega$: letting $X \subseteq \Z$ be the vertex set of $H$, the bijection $\omega \colon \Z \to \Z$ restricts to a bijection $\omega'\colon X \to \omega(X)$ whose inversion graph is $H$, a doubly infinite path. Thus, by Lemma \ref{lem:oshift} (which applies to $\omega'$ because $X$ and $\omega(X)$ are order-isomorphic to $\Z$), $\omega'$ is order-isomorphic to either $\OO$ or $\Sigma \OO$. Either way, the entries of $\omega$ that form $\omega'$ are an occurrence of $\OO$ in $\omega$.

\emph{(b)} Let $\CC$ be an affine permutation class. If $\CC$ is decomposable, then $\OO \not\in \CC$ because $\OO$ is not decomposable. Conversely, if $\CC$ is not decomposable, then there is $\omega \in \CC$ that is not decomposable; by (a), $\omega$ contains $\OO$, so since $\CC$ is a class we have $\OO \in \CC$. This shows that $\CC$ is decomposable if and only if $\OO \not\in \CC$.

Finally, if $\OO \not\in \CC$, then every element of $\CC$ avoids $\OO$, so $\CC \subseteq \AvA(\OO)$; and if $\OO \in \CC$, then $\CC \not\subseteq \AvA(\OO)$ since $\OO \not\in \AvA(\OO)$.

Parts (c) and (d) follow immediately from (b).
\end{proof}

Let $R$ be a set of indecomposable (ordinary) permutations. It follows from Theorem \ref{thm:oscillation} that, if $R$ has an element that is a finite oscillation, then $\oplus \Av(R) = \AvBA(R) = \AvA(R)$. In this case, the problem of enumerating $\AvBA(R)$ or $\AvA(R)$ reduces to the easier problem of enumerating $\oplus \Av(R)$, which is readily accomplished using Theorem \ref{thm:exactclass} (exact) and Theorems \ref{prop:subasympt} and \ref{prop:superasympt} (asymptotic). This reasoning, along with the enumerations found in Example \ref{ex:exact}, recovers a result of Crites \cite{Crites} and proves one of his conjectures:
\begin{cor}
$|\AvA_n(231)| = |\AvA_n(312)| = \binom{2n-1}{n}$ (proved by \cite{Crites}), and $|\AvA_n(3142)| = |\AvA_n(2413)| = \sum_{k=0}^{n-1} \frac{n-k}{n} \binom{n-1+k}{k} 2^k$ (conjectured by \cite{Crites}). \hfill $\square$
\end{cor}

\section{Exact enumeration of all bounded affine permutations}
    \label{sec-first}

We start with the following lemma, whose proof is straightforward.

\begin{lem} \label{lem:stdec} For every affine permutation $\sigma$ of size $n$, there is a unique pair $(\q{\sigma}, a)$ such that:
\begin{verse}  
$\bullet$ $\q{\sigma}\in S_n$;
\\
$\bullet$  $a = a_1 \ldots a_n$ is a word with $a_i \in \Z$ for each $i$, such that $\sum_{i=1}^n a_i = 0$; and
\\
$\bullet$ $\sigma(i) = \q{\sigma}(i) + na_i$ for each $i \in [n]$.
\end{verse}   
Furthermore, $\sigma$ is a bounded affine permutation of size $n$ if and only if
\[ \begin{cases}
a_i \in \{-1,0\} & \text{if $\q{\sigma}(i) > i$}; \\
a_i = 0 & \text{if $\q{\sigma}(i) = i$}; \\
a_i \in \{0,1\} & \text{if $\q{\sigma}(i) < i$}.
\end{cases} \]
\end{lem}

We will call $(\q{\sigma}, a)$ the \emph{standard decomposition} of $\sigma$. The standard decomposition of bounded affine permutations will be useful in obtaining their exact enumeration.   
(Note that $\q{\sigma}$ is different from $\sigma^{\dagger}$ that we defined in the proof of Proposition \ref{prop.3L}.)

An \emph{excedance} of a permutation $\pi \in S_n$ is a position $i \in [n]$ such that $\pi(i) > i$. Let $a(n,k)$ denote the number of permutations in $S_n$ that have $k$ excedances (with $a(0,0) = 1$). These are the Eulerian numbers; see Section 1.4 of Stanley \cite{EC1} for an introduction --- though note that our convention is off by one from his, i.e.\ our $a(n,k)$ is his $A(n,k-1)$.

Let $d(n,k)$ denote the number of derangements (\textit{i.e.}\ permutations with no fixed points) in $S_n$ that have $k$ excedances (with $d(0,0) = 1$); these can be called the derangement Eulerian numbers. Let $d(n)$ denote the number of derangements in $S_n$ (with any number of excedances).

Let $d^{(m)}(n,k)$ denote the number of permutations in $S_n$ that have $k$ excedances and $m$ fixed points (with $d^{(m)}(0,0) = 0$ for $m\not=0$). We have $a(n,k) = \sum_{m=0}^{n-k} d^{(m)}(n,k)$ and $d(n,k) = d^{(0)}(n,k)$. Let $d^{(m)}(n)$ denote the number of permutations in $S_n$ that have $m$ fixed points (with any number of excedances).

The following lemma is a classical result; part (b) is a classic application of the Inclusion--Exclusion Principle.

\begin{lem}
      \label{lem:dclassic}
\begin{enumerate}[(a)]
\item $d^{(m)}(n) = \binom{n}{m} d(n-m)$.
\item $d(n) = \sum_{m=0}^n \binom{n}{m} (-1)^m (n-m)! = n! \sum_{m=0}^n \frac{(-1)^m}{m!}$.
\item $d(n) \sim n!/e$
\end{enumerate}
\end{lem}

The next lemma is a refinement of Lemma \ref{lem:dclassic}
according to excedance number.

\begin{lem} \label{lem:exced}
\begin{enumerate}[(a)]
\item $d^{(m)}(n,k) = \binom{n}{m} d(n-m,k)$.
\item $d(n,k) = \sum_{m=0}^{n-k} \binom{n}{m} (-1)^m a(n-m,k)$.
\end{enumerate}
\end{lem}

\begin{proof}
(a) Let $\pi$ be a permutation counted by $d(n-m,k)$; that is, $\pi$ is a derangement in $S_{n-m}$ with $k$ excedances. Let $I \subseteq [n]$ be a set of size $m$. Then $\pi$ and $I$ give rise to a permutation $\pi'$ in $S_n$ as follows: the elements of $I$ are the fixed points of $\pi'$, and $\pi'$ permutes the remaining elements of $[n]$ according to $\pi$. This map $(\pi, I) \mapsto \pi'$ is bijective onto the set of permutations in $S_n$ with $m$ fixed points and $k$ excedances.

(b) $a(n,k) = \sum_{m=0}^{n-k} d^{(m)}(n,k)$; thus, by (a), $a(n,k) = \sum_{m=0}^{n-k} \binom{n}{m} d(n-m,k)$. The result follows by the Inclusion--Exclusion Principle.
\end{proof}

\begin{thm} \label{thm:exacttotal} The number of bounded affine permutations of size $n$ is
\begin{enumerate}[(a)]
\item $|\E_n| = \sum_{m=0}^n \binom{n}{m} \sum_{k=0}^m \binom{m}{k} d(m,k)$;
\item $|\E_n| = \sum_{m=0}^n \binom{n}{m} \sum_{k=0}^m \binom{m}{n-k} (-1)^{n-m} a(m,k)$.
\end{enumerate}
\end{thm}

\begin{proof}
We use Lemma \ref{lem:stdec}. Given $\q{\sigma} \in S_n$ with $m$ fixed points and $k$ excedances, how many words $a \in \{0, 1, -1\}^n$ make $(\q{\sigma}, a)$ the standard decomposition of a bounded affine permutation? Let $r$ be the number of $1$'s in $a$; then $r$ is also the number of $(-1)$'s in $a$. The $1$'s occur at positions $i$ where $\sigma(i) < i$; there are $n-m-k$ such positions, so the number of ways of placing $1$'s is $\binom{n-m-k}{r}$. Similarly, the $(-1)$'s occur at positions $i$ where $\sigma(i) > i$; there are $k$ such positions, so the number of ways of placing $(-1)$'s is $\binom{k}{r}$. The other entries of $a$ are all $0$. Thus, the number of ways to choose $a$ with $r$ $1$'s, for a given $\q{\sigma}$, is $\binom{n-m-k}{r} \binom{k}{r}$.

Summing over all $r$, the number of ways to choose $a$ with any number of $1$'s, for a given $\q{\sigma}$, is $\sum_{r\ge0} \binom{n-m-k}{r} \binom{k}{r}$. This equals $\binom{n-m}{k}$, a fun exercise for an introductory combinatorics class. Thus, summing over all $\q{\sigma}$ and all $m$ and $k$,
\[ |\E_n| = \sum_{m=0}^n \sum_{k=0}^{n-m} \binom{n-m}{k} d^{(m)}(n,k). \]
Using Lemma \ref{lem:exced}, this becomes
\begin{align*}
|\E_n| &= \sum_{m=0}^n \binom{n}{m} \sum_{k=0}^{n-m} \binom{n-m}{k} d(n-m,k) \\
&= \sum_{m=0}^n \binom{n}{m} \sum_{k=0}^m \binom{m}{k} d(m,k),
\end{align*}
proving (a).

Again using Lemma \ref{lem:exced},
\begin{align*}
\sum_{m=0}^n \binom{n}{m} \sum_{k=0}^m \binom{m}{k} d(m,k) &= \sum_{m=0}^n \sum_{k=0}^m \sum_{m'=0}^{m-k} \binom{n}{m} \binom{m}{k} \binom{m}{m'} (-1)^{m'} a(m-m',k) \\
&= \sum_{p=0}^n \sum_{k=0}^p \left[ \sum_{m=p}^n \binom{n}{m} \binom{m}{k} \binom{m}{p} (-1)^{m-p} \right] a(p,k) \\
&= \sum_{p=0}^n \sum_{k=0}^p \binom{n}{p} \binom{p}{n-k} (-1)^{n-p} a(p,k),
\end{align*}
proving (b). (Simplifying the summation to arrive at the last line is a more advanced combinatorics exercise.)
\end{proof}

Part (b) of Theorem \ref{thm:exacttotal} expresses $|\E_n|$ in terms of Eulerian numbers $a(m,k)$. As these numbers are well-known, this equality can be taken as an ``answer'' to how many bounded affine permutations there are. Part (a) expresses $|\E_n|$ in terms of ``derangement Eulerian numbers'' $d(m,k)$. These numbers, though easily obtained by Lemma \ref{lem:exced}, are not so canonical; but part (a) of the theorem will be used in determining the asymptotic enumeration of $|\E_n|$.

\section{Asymptotic enumeration of all bounded affine permutations}
   \label{sec-asymp}

\subsection{Gaussian local limit laws}

To find an asymptotic formula for $|\E_n|$, we need to understand the numbers $d(n,k)$ better; to do that, we need to introduce the concepts of \emph{convergence in distribution} and \emph{local limit law}, in the special case where the limiting distribution is a Gaussian distribution.  We follow \cite{FS}.

\begin{defn}
(a) Let $\{f_n(k):k\in \Z, n\in\mathbb{N}\}$ 
be nonnegative real numbers.  
For each $n$, let $f^*_n:=\sum_kf_n(k)$ and assume that
$f^*_n$
is nonzero and finite.  Then for each $n$ we can associate 
a probability distribution $p_n(\cdot)$ on $\Z$ by 
$p_n(k):=f_n(\cdot)/f^*_n$ ($k\in\Z$).  
Let $\mu_f(n)$ and $\sigma_f(n)$ denote the mean and the 
standard deviation of this probability distribution.
\\
(b)  Under the assumptions of (a), we say that $p_n(k)$ (or $f_n(k)$)
\emph{converges in distribution when standardized 
to a Gaussian distribution} if, for each $x \in \R$,
\[ \lim_{n\to\infty} \sum_{k \,\le\, k(n,x)} \frac{f_n(k)}{f^*_n} = \int_{-\infty}^x \frac{1}{\sqrt{2\pi}}\, e^{-t^2/2}\,dt \,, 
\]
where we defined
\[ k(n,x) = \left\lfloor \mu_f(n) + x\sigma_f(n) \right\rfloor. \]
Observe that the sum on the left is the cumulative 
distribution function of $(X_n-\mu_f(n))/\sigma_f(n)$ 
where $X_n$ has distribution $p_n(\cdot)$.
\\
(c)  Under the assumptions of (a), we say that $p_n(k)$ (or $f_n(k)$)
\emph{obeys a Gaussian local limit law} if for each $x\in\R$ and for $k(n,x)$
as in part (b), we have
\[ \lim_{n\to\infty} \sigma_n \frac{f_n\left(k(n,x)\right)}{f^*_n} =
\frac{1}{\sqrt{2\pi}} \, e^{-x^2/2} \]
and the convergence is uniform in $x$. Equivalently, 
with $k = k(n,x)$ as above,
\[ f_n(k) = \frac{1}{\sqrt{2\pi}\,\sigma_f(n)} f^*_n 
 \left[ e^{-x^2/2} + o(1) \right] \]
where the error term $o(1)$ goes to $0$ uniformly in $x$.

\end{defn}

Parts (b) and (c) of the above definition say that the numbers $f_n(k)$ $(k\in \mathbb{Z})$ approximate a normal curve, in two particular senses.

The formula in Theorem \ref{thm:exacttotal}(a) includes a sum of terms of the form $\binom{n}{k} d(n,k)$; we will see later that both $\binom{n}{k}$ and $d(n,k)$ have Gaussian limits. The following result will allow us to deal with the sum.

\begin{lem}
    \label{lem.Gaussian}
Let $\{\alpha_n(k),\beta_n(k)\,:\,n\in\mathbb{N},k\in\mathbb{Z}\}$ be nonnegative numbers.
Let $\alpha^*_n \,=\, \sum_k \alpha_n(k)$ and   $\beta^*_n\,=\,  \sum_k \beta_n(k)$, and assume that
these are nonzero and finite, so that for 
each $n$ we can define probability distributions 
on $\mathbb{Z}$ by $\alpha_n(\cdot)/\alpha^*_n$ 
and $\beta_n(\cdot)/\beta_n^*$.
For 
$B_1,B_2\in (0,\infty)$, 
assume that as $n\rightarrow\infty$, 
\\
(\textit{i})  The distributions $\alpha_n(\cdot)/\alpha^*_n$
have means $\mu_{\alpha}(n)$ and 
standard deviations $\sigma_{\alpha}(n) \sim \sqrt{B_1n}$ and converge in distribution when standardized to
a Gaussian distribution; and
\\
(\textit{ii})  The distributions $\beta_n(\cdot)/\beta^*_n$ have means $\mu_{\beta}(n)$ and 
standard deviations $\sigma_{\beta}(n) \sim\sqrt{ B_2n}$ and obeys a Gaussian local limit law; and
\\
(\text{iii})   $\mu_{\alpha}(n)-\mu_{\beta}(n) \,=\,o(\sqrt{n})$.
\\
Then 
\[      \sum_k  \alpha_n(k)\beta_n(k)    \;\sim\;   \frac{\alpha_n^* \beta_n^*}{\sqrt{2\pi  (B_1+B_2)\,n}}
    \hspace{5mm}  \hbox{ as }n\rightarrow\infty.
\]
\end{lem}

\begin{proof}
Define the functions $\phi$ and $b_n$  ($n\in \mathbb{N}$) from $\mathbb{R}$ to $[0,\infty)$ by
\begin{align*}
     \phi(x)   \;  & = \;   \frac{1}{\sqrt{2\pi}} \, e^{-x^2/2}   \,,  \\
     b_n(x)   \; & = \;         \beta_n(\lfloor x  \rfloor)    \,.
\end{align*}
Then assumption (\textit{ii}) says that $\lim_{n\rightarrow\infty}\epsilon_n = 0$, where
\begin{equation}
   \label{eq.epsmdef1}
   \epsilon_n  \;:=\;   \sup_{x\in \mathbb{R}} 
    \left|   \frac{ b_n\! \left( \mu_{\beta}(n)+x\sigma_{\beta}(n)\right)}{\beta^*_m} \,\sigma_{\beta}(n) 
         \;-\; \phi(x) \right|  \,.
\end{equation}

For each positive integer $n$, let $Z_n$ be a random variable whose probability distribution is
\[     \Pr(Z_n = k)   \;=\;   \frac{\alpha_n(k)}{\alpha^*_n}  \hspace{15mm}(k\in\mathbb{Z}).    \]
Assumption (\textit{i}) says  the standardized random variable
\[     W_n    \;  := \;  \frac{Z_n-\mu_{\alpha}(n)}{\sigma_{\alpha}(n)}     \]
converges in distribution to the standard normal distribution as $n\rightarrow\infty$.

Define
\[     \theta_n \;=\;  \sum_{k\in\mathbb{Z}}   \frac{\alpha_n(k)}{\alpha^*_n}\, \beta_n(k)   \,.    \]
Then we can express $\theta_n$ as an expected value of a function of $Z_n$:
\begin{align}
    \nonumber
    \theta_n   \;  & =  \;  E\left[ b_n(Z_n)\right]    \\
       \nonumber
        & =\;     E\left[ b_n\! \left(   \mu_{\alpha}(n) +W_n\sigma_{\alpha}(n) \right) \right]   \\
          \label{eq.thetaEb1}
         & =\;     E\left[ b_n\! \left(      \mu_{\beta}(n) +
       \widehat{W}_n   \sigma_{\beta}(n)        \right) \right]  
\end{align}
where 
\[     \widehat{W}_n   \; :=\;   W_n\left(\frac{\sigma_{\alpha}(n)}{\sigma_{\beta}(n)}\right) +  
        \frac{\mu_{\alpha}(n)- \mu_{\beta}(n)}{\sigma_{\beta}(n)} \,.
\]  
Let $W_{\infty}$ be a standard normal random variable, so that $W_n$ converges in distribution to $W_{\infty}$.              
By our assumptions, we know that $\sigma_{\alpha}(n)/\sigma_{\beta}(n)\rightarrow \sqrt{B_1/B_2}$
and $| \mu_{\alpha}(n)- \mu_{\beta}(n)|/\sigma_{\beta}(n)\rightarrow 0$ as $n\rightarrow\infty$.
Therefore, by the corollary to Theorem 4.4.6 in \cite{Chung},
we see that $\widehat{W}_n$ converges in distribution to $W_{\infty}\sqrt{B_1/B_2}$.

Now define
\[      \psi_n    \;:=\;    E\left[ \phi\left(     \widehat{W}_n    \right)\right] \,.  \]
Then, by Equations (\ref{eq.epsmdef1}) and (\ref{eq.thetaEb1}), we see that
\begin{equation}
     \label{eq.thetapsibd1}
        \left|   \frac{\theta_n \,\sigma_{\beta}(n)}{ \beta^*_n}  \;-\; \psi_n \right|    \;\leq \; \epsilon_n  \,.
\end{equation}     

A fundamental property of convergence in distribution is that if a  sequence of random variables $\{X_m\}$
converges in distribution to $X$, then $E[g(X_m)]$ converges to $E[g(X)]$ for every bounded
continuous function $g$ (see for example Theorem 2.1 of Billingsley \cite{Bill}; 
indeed, like many 
authors, Billingsley uses this property as the definition of convergence in distribution).
Thus, writing $c=\sqrt{B_1/B_2}$, we obtain
\begin{align*}
   \lim_{n\rightarrow\infty}  \psi_n   \; & = \; E\left[ \phi( c W_{\infty} )\right]    \\
       & = \;   \int_{-\infty}^{\infty}   \left(  \frac{1}{\sqrt{2\pi}} \, e^{-c^2x^2/2} \right)   \,
              \frac{1}{\sqrt{2\pi}} \, e^{-x^2/2} \,dx    \\
       & = \;   \frac{1}{2\pi}    \int_{-\infty}^{\infty}   e^{-(1+c^2)x^2/2} \,dx    \\
       & = \;  \frac{1}{2\pi} \sqrt{ \frac{2\pi}{1+c^2} }  
           \hspace{15mm}\left( \mbox{using } \int_{-\infty}^{\infty}e^{-Ax^2}\,dx =\sqrt{\pi/A}\right) \\
        & = \;   \sqrt{    \frac{B_2}{2\pi (B_1+B_2)} } \,.
\end{align*}
Applying this to Equation (\ref{eq.thetapsibd1}), and using the convergence of $\epsilon_n$ to 0, we obtain
\[     \lim_{n\rightarrow\infty}  \frac{\theta_n}{\beta^*_n} \, \sigma_{\beta}(n)    \;=\;  
       \, \sqrt{    \frac{B_2}{2\pi (B_1+B_2)} } \,.
\] 
Finally, since $\sigma_{\beta}(n)\sim \sqrt{B_2n}\,$, the lemma follows.
\end{proof}

\subsection{Derangement Eulerian numbers}

Define
\[ A(z,u) = \sum_{n,k\ge0} a(n,k) u^k \frac{z^n}{n!} \quad \text{and} \quad D(z,u) = \sum_{n,k\ge0} d(n,k) u^k \frac{z^n}{n!}. \]
So $A(z,u)$ (resp.\ $D(z,u)$) is the bivariate generating function counting permutations (resp.\ derangements) according to excedance number. It is a classical result that
\[ A(z,u) = \frac{u-1}{u-e^{(u-1)z}}. \]
A permutation is a derangement with an added set of fixed points, and inserting additional fixed points into a permutation preserves the permutation's excedance number. Thus, since $e^z$ is the generating function counting sets of fixed points,
\[ A(z,u) = e^z D(z,u), \]
and so
\[ D(z,u) = e^{-z} A(z,u) = \frac{(u-1)e^{-z}}{u-e^{(u-1)z}}. \]
(This also serves as a much more compact proof of Lemma \ref{lem:exced}.)

The following theorem was proved by Bender \cite{Bender} and appears as Example
IX.35 in in Flajolet and Sedgewick \cite{FS} (though note that these use the alternative convention for the Eulerian numbers, used also by Stanley \cite{EC1}).
\begin{thm} \label{thm:gaussianA} The Eulerian numbers $a(n,k)$ obey a Gaussian local limit law with mean $\mu_a(n)= (n-1)/2$ and variance $\sigma_a(n)^2 \sim n/12$.
\end{thm}

The proof in Flajolet and Sedgewick uses the analytic properties of $A(z,u)$ treated as a function of complex variables. We will use that idea to obtain:
\begin{thm} \label{thm:gaussianD} 
The derangement Eulerian numbers $d(n,k)$ 
obey a Gaussian local limit law
with mean $\mu_d(n) = n/2$ and variance $\sigma_d(n)^2 \sim n/12$.
\end{thm}

We need the following theorem, which comes from Theorems IX.9 and IX.14
in Flajolet and Sedgewick \cite{FS}.
\begin{thm} \label{lem:FS} Let $F(z,u)$ be a function that is bivariate analytic at $(z,u) = (0,0)$ and has non-negative coefficients of its power series. Let $r>0$ such that $F(z,1)$ is meromorphic on the disk $|z| \le r$. Assume the following:
\begin{enumerate}[(1)]
\item The only singularity of $F(z,1)$ on the disk $|z| \le r$ is a simple pole at $z = \rho$ for some $\rho \in (0, r)$.
\item There are functions $B(z,u)$ and $C(z,u)$, analytic for $z$ with $|z| \le r$ and $u$ in some neighborhood $1$, such that
\[ F(z,u) = \frac{B(z,u)}{C(z,u)} \]
and $B(\rho,1) \not= 0$.
\item There is non-constant $\rho(u)$, analytic at $u=1$, such that $C(\rho(u), u) = 0$ and $\rho(1) = \rho$.
\item $v : = \left( \frac{\rho'(1)}{\rho(1)} \right)^2 - \frac{\rho'(1)}{\rho(1)} - \frac{\rho''(1)}{\rho(1)} \not= 0$.
\item For every $u$ with $|u| = 1$, the function $z \mapsto F(z,u)$ has no singularity with modulus $\le \rho$ (unless $u=1$).
\end{enumerate}
Also set $m = -\frac{\rho'(1)}{\rho(1)}$. Then the numbers $f(n,k) := [z^n u^k]F(z,u)$ satisfy a local limit law of Gaussian type with mean $\mu_f(n) \sim mn$ and standard deviation $\sigma_f(n) \sim \sqrt{vn}$.
\end{thm}

\begin{proof}
Conditions (1) through (4) correspond in a straightforward manner to the conditions in Theorem IX.9 from Flajolet and Sedgewick (with $\mathfrak{v}$ defined in Equation (27) on page 645). Hence the numbers $f(n,k)$ converge \emph{in distribution} to a Gaussian distribution. The mean $\mu_f(n)$ and variance $\sigma_f(n)$ of this distribution are given by Equation (37) in the first paragraph of Flajolet and Sedgewick's proof:
\[ \mu_f(n) \sim \mathfrak{m}\!\left( \frac{\rho(1)}{\rho(u)} \right) n \quad \text{and} \quad \sigma_f(n)^2 \sim \mathfrak{v}\!\left( \frac{\rho(1)}{\rho(u)} \right) n, \]
where $\mathfrak{m}$ and $\mathfrak{v}$ are operators defined by Flajolet and Sedgewick in their Equation (27) on page 645. Using the definitions of $\mathfrak{m}$ and $\mathfrak{v}$ yields $\mu_f(n) \sim mn$ and $\sigma_f(n)^2 \sim vn$, as claimed.

To show that the numbers $f(n,k)$ satisfy a local limit law of Gaussian type, we use Theorem IX.14 from Flajolet and Sedgewick. This theorem requires that $p_n(u) := \frac{[z^n]F(z,u)}{[z^n]F(z,1)}$ satisfies the conditions of the ``Quasi-Powers Theorem'' (Theorem IX.8 from Flajolet and Sedgewick) and an additional condition. The proof of Theorem IX.9 shows that $p_n(u)$ satisfies the conditions of the Quasi-Powers Theorem, and the additional condition is implied by our (5), as explained on page 697 of Flajolet and Sedgewick (the paragraph after the proof of Theorem IX.14).
\end{proof}

For our purposes, we can safely ignore Theorem \ref{lem:FS} after we use it to obtain the following:

\begin{lem} \label{lem:same}
Assume $F(z,u)$ satisfies the conditions of Theorem \ref{lem:FS}, and let $\alpha(z,u)$ be a function that is analytic and non-zero on all of $\C \times \C$. Assume also that $\alpha(z,u)\,F(z,u)$ has non-negative coefficients of its power series. Then $\alpha(z,u)\,F(z,u)$ satisfies the conditions of Theorem \ref{lem:FS}, and the mean and standard deviation of the coefficients of $\alpha(z,u)\,F(z,u)$ are asymptotically the same as for $F(z,u)$.
\end{lem}

\begin{proof}
Because $\alpha(z,u)$ is analytic and non-zero everywhere, $\alpha(z,u)\,F(z,u)$ has singularities at the same points and with the same orders as $F(z,u)$. Thus, we can take $\rho$ and $r$ to be the same, we can replace $B(z,u)$ with $\alpha(z,u)\,B(z,u)$, and we can take $C(z,u)$ and $\rho(u)$ to be the same. This shows that $\alpha(z,u)\,F(z,u)$ meets conditions (1) through (4) of Theorem \ref{lem:FS}. It also meets condition (5), since for each $u$ the singularities of $z \mapsto F(z,u)$ are identical to the singularities of $z \mapsto \alpha(z,u)\,F(z,u)$.

Finally, $m$ and $v$ depend only on $\rho$, so $m$ and $v$ are the same for $\alpha(z,u)\,F(z,u)$, and so the resulting mean and standard deviation are asymptotically the same.
\end{proof}

\begin{proof}[of Theorem \ref{thm:gaussianD}]
The proof that $\mu_d(n) = n/2$ is elementary: Let $\pi$ be a derangement of size $n$. For each $i \in [n]$, $i$ is an excedance of $\pi$ if and only if $\pi(i)$ is not an excedance of $\pi^{-1}$, precisely because $\pi$ has no fixed points. Thus the map $\pi \mapsto \pi^{-1}$ is an involution on the set of size-$n$ derangements that maps each derangement with $k$ excedances to a derangement with $n-k$ excedances. Therefore, $d(n,k) = d(n,n-k)$, and so the mean is $n/2$.

By Flajolet and Sedgewick \cite{FS}, Examples IX.12 and IX.35,
$A(z,u)$ satisfies the conditions of Theorem \ref{lem:FS}. Since $D(z,u) = e^{-z} A(z,u)$ and $e^{-z}$ is analytic and non-zero everywhere, it follows from Lemma \ref{lem:same} that $D(z,u)$ also satisfies the conditions of Theorem \ref{lem:FS}, and that the coefficients have the same asymptotic mean and standard deviation as those of $A(z,u)$.
\end{proof}

\begin{thm}
   \label{thm:Enasym}
$|\E_n| \sim \sqrt{\frac{3}{2\pi e}} \,n^{-1/2}\, 2^n n!$.
\end{thm}

\begin{proof}
Using the exact enumeration from Theorem \ref{thm:exacttotal}, we have
\begin{align}
    \nonumber
     \frac{  |\E_n| \,\sqrt{n}}{n!\,2^n} \; & = \;  
     \sum_{m=0}^n \binom{n}{m}  \frac{\sqrt{n}}{n!\,2^n}\sum_{k=0}^m \binom{m}{k} d(m,k) 
  \\   
  \nonumber
  & = \; \sqrt{n} \,\sum_{m=0}^n \frac{1}{m!(n-m)!} \,\frac{1}{2^{n-m}} \,\frac{ \sum_{k=0}^m \binom{m}{k}d(m,k)    
      }{2^m} 
  \\
 \label{eq.Ensum}
    & = \;  \sqrt{n}\, \sum_{w=0}^n \frac{2^{-w}}{w!} \,Q_{n-w}    \hspace{5mm}(\hbox{using $w=n-m$})
\end{align}
where we defined
\[   Q_m  \;:=\;    \frac{\sum_{k=0}^m \binom{m}{k}d(m,k)}{2^m m!} \,.
\]

We now apply Lemma \ref{lem.Gaussian} with 
$\alpha_n(k) =\binom{n}{k}$ and $\beta_n(k)=d(n,k)$.  Then we have
$\alpha_n^*=2^n$, $\mu_{\alpha}(n)=n/2$, and $\sigma_{\alpha}(n)=\sqrt{n/4}$, as well as 
$\beta_n^*=d(n)\sim n!/e$, $\mu_{\beta}(n)=n/2$, and
$\sigma_{\beta}(n)\sim \sqrt{n/12}$.  Assumption (\textit{i}) of the lemma, the convergence in 
distribution of $\alpha_n(k)$, is the classical
De Moivre--Laplace Central Limit Theorem (e.g.\ p.\ 186 of
\cite{Fell1}).   Assumption (\textit{ii}), 
the local limit law of $\beta_n(k)$, is Theorem \ref{thm:gaussianD} above.
The conclusion of Lemma \ref{lem.Gaussian} is
\begin{equation}
   \label{eq.Qlim}
       \lim_{m\rightarrow\infty}\sqrt{m}\, Q_m  \;=\;
       \frac{1}{e}\,\sqrt{  \frac{3}{2\pi} }   \,.   
\end{equation}
In particular, the sequence $\{\sqrt{m}\,Q_m\}$ is bounded.  Let
\[     Q^*  \;  :=\;   \sup_{m\in \mathbb{N}}  \left\{\sqrt{m}\,Q_m \right\}  \;<\; \infty  \,.   \]

We now break the sum of Equation (\ref{eq.Ensum}) into two parts, writing
\begin{equation}
     \label{eq.EAB}
      \frac{|\E_n|\sqrt{n}}{n!\,2^n}    \;=\;  C_n  \,+\,D_n \, ,  
\end{equation}
\[  \hbox{where} \hspace{15mm}  C_n  \;=\;    \sqrt{n}\, \sum_{w=0}^{\lfloor n/2 \rfloor} \frac{2^{-w}}{w!} \,Q_{n-w} 
     \hspace{5mm}\hbox{and}\hspace{5mm}
     D_n \;=\;     \sqrt{n}\, \sum_{w=\lfloor n/2\rfloor+1}^n \frac{2^{-w}}{w!} \,Q_{n-w} \,.
\]
Then 
\[    0\;\leq \; D_n  \; \leq  \;   \sqrt{n}\,Q^* \sum_{w=\lfloor n/2\rfloor+1}^n 2^{-w}
       \;<\;   \sqrt{n}\,Q^* \,2^{-\lfloor n/2 \rfloor}   \,,
\]
from which we see that
\begin{equation}
     \label{eq.Bto0}
         \lim_{n\rightarrow\infty}  D_n   \;=\; 0 \,.    
\end{equation}
To find the limit of $C_n$, we shall use the Dominated Convergence Theorem with respect to
the (finite) measure on the nonnegative integers ${\mathbb Z}_+$
that gives mass $2^{-w}/w!$ to each $w\in {\mathbb Z}_+$.
Define the sequence of functions $g_n$ on ${\mathbb Z}_+$ by
\[     g_n(w)   \; =\;   \begin{cases}
            \sqrt{n}   \,Q_{n-w}    & \hbox{if   $0\leq w \leq \lfloor n/2\rfloor$}   \\
            0   & \hbox{if   $w>\lfloor n/2\rfloor$}.
            \end{cases}
\]
Thus 
\begin{equation}
    \label{eq.Ansumfn}
       C_n  \,=\,  \sum_{w=0}^{\infty}\frac{2^{-w}}{w!}  \,g_n(w)   \,.
\end{equation}
Let $g$ be the constant function 
$g(w)\equiv \frac{1}{e}\,\sqrt{  \frac{3}{2\pi} }$\,.
By Equation (\ref{eq.Qlim}), $\lim_{n\rightarrow\infty}g_n(w)   \,=\,  g(w)$
for every $w\in \mathbb{Z}_+$.
Moreover, we claim that
\begin{equation}
  \nonumber
          g_n(w)   \;\leq \;   \sqrt{2} \,Q^*  
              \hspace{5mm}\hbox{for every $n\in \mathbb{N}$ and $w\in \mathbb{Z}_+$} \,. 
\end{equation}
This inequality   is trivial if $w>\lfloor n/2\rfloor$, 
while for $0\leq w \leq \lfloor n/2\rfloor$ we have
\[     g_n(w)   \;\leq  \;   \frac{\sqrt{n}\,Q^*}{\sqrt{n-w}}    \;\leq \;    \frac{\sqrt{n}\,Q^*}{\sqrt{n-(n/2)}} 
     \;=\;   \sqrt{2}\,Q^* \,.
\]
This proves the claim.  

We have shown that the functions $g_n$ are uniformly bounded and converge
pointwise to $g$.  Therefore we can apply the Dominated Convergence Theorem:
\begin{align}
    \nonumber
     \lim_{n\rightarrow\infty}C_n    \; & =\;  \lim_{n\rightarrow\infty}  \sum_{w=0}^{\infty}\frac{2^{-w}}{w!}  \,g_n(w) \\
        \nonumber
          & = \;    \sum_{w=0}^{\infty}\frac{2^{-w}}{w!}  \,g(w)   \\
          \nonumber
        & = \;    \sum_{w=0}^{\infty}\frac{2^{-w}}{w!} \frac{1}{e}\,\sqrt{  \frac{3}{2\pi} }    \\
           \label{eq.limAn}
         & = \;    \frac{e^{1/2}}{e}\,\sqrt{  \frac{3}{2\pi} }  \,.
\end{align}          
Combining Equations (\ref{eq.EAB}), (\ref{eq.Bto0}), and (\ref{eq.limAn})  completes the proof of 
Theorem \ref{thm:Enasym}.
\end{proof}

\acknowledgements
\label{sec:ack}
We would like to thank Nathan Clisby for the initial suggestion 
of using periodic boundary conditions to study pattern avoidance in ordinary permutations, 
Cyril Banderier for pointing us in the direction of the cycle lemma and its orbit of ideas, and David Bevan for helpful discussions about supercritical and subcritical sequence schemas. We would also like to thank our anonymous referee for carefully reading our paper and providing helpful recommendations.


\end{document}